\documentclass{amsart}

\newtheorem{theorem}{Theorem}[section]

\newtheorem{corollary}[theorem]{Corollary}
\newtheorem{definition}[theorem]{Definition}
\newtheorem{example}[theorem]{Example}
\newtheorem{lemma}[theorem]{Lemma}
\newtheorem{proposition}[theorem]{Proposition}
\theoremstyle{remark}
\newtheorem{remark}[theorem]{Remark}
\numberwithin{equation}{section}

\begin{document}
\title[Bi-paracontact structures and Legendre foliations]{Bi-paracontact structures and Legendre foliations}

\author[B. Cappelletti Montano]{Beniamino Cappelletti Montano}

\address{Dipartimento di Matematica,
Universit\`{a} degli Studi di Bari, Via E. Orabona 4, 70125 Bari,
Italy}

\email{b.cappellettimontano@gmail.com}

\subjclass[2000]{}


\begin{abstract}
We study almost bi-paracontact structures on contact manifolds. We
prove that if an almost bi-paracontact structure is defined on a
contact manifold $(M,\eta)$, then under some natural assumptions of
integrability, $M$ carries two transverse bi-Legendrian structures.
Conversely, if two transverse bi-Legendrian structures are defined
on a contact manifold, then $M$ admits an almost bi-paracontact
structure. We define a canonical connection on an almost
bi-paracontact manifold and we study its curvature properties, which
resemble those of the Obata connection of an anti-hypercomplex (or
complex-product) manifold. Further, we prove that any contact metric
manifold whose Reeb vector field belongs to the
$(\kappa,\mu)$-nullity distribution canonically carries an almost
bi-paracontact structure and we apply the previous results to the
theory of contact metric $(\kappa,\mu)$-spaces.
\end{abstract}

\maketitle

\section{Introduction}

The study of Legendre foliations on contact manifolds is very recent
in literature, being initiated in the early 90's by the work of
Libermann, Pang et alt. (cf. \cite{libermann}, \cite{pang}). Lately,
the notion of ``bi-Legendrian'' structure has made its appearance,
especially with regard to its applications to Cartan geometry
(\cite{morimoto1}) and Monge-Amp\`{e}re equations (\cite{morimoto2})
and to other geometric structures associated with a contact
manifold, such as paracontact metrics. In particular, in
\cite{Mino-08} the author studied the interplays between
bi-Legendrian manifolds and paracontact geometry, whereas in
\cite{Mino-toap1} the theory of bi-Legendrian structures was applied
for the study of a remarkable class of contact Riemannian manifolds,
namely contact metric $(\kappa,\mu)$-spaces. We recall that a
contact metric $(\kappa,\mu)$-space is a contact Riemannian manifold
$(M,\phi,\xi,\eta,g)$ such that the Reeb vector field $\xi$ belongs
to the $(\kappa,\mu)$-nullity distribution, i.e. the following
condition holds
\begin{equation*}
R^{g}(X,Y)\xi=\kappa\left(\eta\left(Y\right)X-\eta\left(X\right)Y\right)+\mu\left(\eta\left(Y\right)hX-\eta\left(X\right)hY\right),
\end{equation*}
for some real numbers $\kappa$, $\mu$ and for  any
$X,Y\in\Gamma(TM)$, where $R^{g}$ denotes the curvature tensor field
of the Levi Civita connection and $2h$ is the Lie derivative of the
structure tensor $\phi$ in the direction of the Reeb vector field.
This definition, which has no analogue in even dimension, was
introduced by Blair, Kouforgiorgos and Papantoniou in \cite{BKP-95},
as a generalization both of the well-known Sasakian condition
$R^{g}(X,Y)\xi=\eta\left(Y\right)X-\eta\left(X\right)Y$ and of those
contact metric manifolds verifying $R^{g}(X,Y)\xi=0$ which were
studied by Blair in \cite{blair-1}. A notable class of examples of
contact metric $(\kappa,\mu)$-spaces is given by the tangent sphere
bundle of Riemannian manifolds of constant curvature.

One of the main results in \cite{BKP-95} was that any non-Sasakian
contact metric $(\kappa,\mu)$-space is foliated by two mutually
orthogonal Legendre foliations ${\mathcal D}_{h}(\lambda)$ and
${\mathcal D}_{h}(-\lambda)$, given by the eigendistributions of the
symmetric operator $h$ corresponding to the eigenvalues $\lambda$
and $-\lambda$, respectively, where $\lambda:=\sqrt{1-\kappa}$. Thus
any contact metric $(\kappa,\mu)$-space is canonically a
bi-Legendrian manifold.

In this paper we show that this is only a part of the story. In fact
we prove that also the operator $\phi h$ is diagonalizable and
admits the same eigenvalues as $h$. Overall, the corresponding
eigendistributions ${\mathcal D}_{\phi h}(\lambda)$ and ${\mathcal
D}_{\phi h}(-\lambda)$ are integrable and define two mutually
orthogonal Legendre foliations, as well. Thus any contact metric
$(\kappa,\mu)$-space carries two bi-Legendrian structures and,
moreover, any foliation of each bi-Legendrian structure is
transversal to the foliations of the other one. This geometrical
structure resembles the concept, in even dimension, of
\emph{$3$-web} (\cite{nagy}) together with its closely linked
tensorial notion, \emph{anti-hypercomplex or complex-product
structure} (\cite{andrada}, \cite{marchiafava}). In fact, let
$\phi_{1}$, $\phi_{2}$, $\phi_{3}$ denote the (1,1)-tensor fields
defined by
\begin{equation}\label{kappamu}
\phi_{1}:=\frac{1}{\sqrt{1-\kappa}}\phi h, \ \ \
\phi_{2}:=\frac{1}{\sqrt{1-\kappa}}h, \ \ \ \phi_{3}:=\phi.
\end{equation}
Then one can check that $\phi_{1}$ and $\phi_{2}$ are anti-commuting
almost paracontact structures on $M$ such that
$\phi_{1}\phi_{2}=\phi_{3}$.

Thus we are motivated in the study of this new geometric structure,
which we call \emph{almost bi-paracontact structure}. An almost
bi-paracontact structure on a contact manifold $(M,\eta)$ is by
definition any triplet $(\phi_{1},\phi_{2},\phi_{3})$, where
$\phi_{1}$ and $\phi_{2}$ are anti-commuting tensor fields
satisfying $\phi_{1}^{2}=\phi_{2}^{3}=I-\eta\otimes\xi$ and
$\phi_{3}=\phi_{1}\phi_{2}$ is an almost contact structure on
$(M,\eta)$. Then one can prove that $\phi_{1}$ and $\phi_{2}$ are in
fact almost paracontact structures and the eigendistributions
corresponding to $\pm 1$ define, under some natural assumptions,
four mutually transversal Legendre foliations.

When the structure is \emph{normal}, that is when the Nijenhuis
tensors of $\phi_{1}$, $\phi_{2}$, $\phi_{3}$ vanish, the leaves of
such foliations admit an affine structure. This is due to the
existence of a unique linear connection $\nabla^{c}$ which preserves
$\phi_{1}$, $\phi_{2}$, $\phi_{3}$. $\nabla^{c}$ is called the
\emph{canonical connection} of the almost bi-paracontact manifold
$(M,\phi_{1},\phi_{2},\phi_{3})$ and it can be considered, in some
sense, as the odd-dimensional counterpart of the Chern connection of
an almost anti-hypercomplex manifold (\cite{marchiafava}), as well
as of the connection studied by Andrada for a complex-product
manifold (\cite{andrada}), and of the Obata connection of a manifold
endowed with an almost quaternion structure of the second kind
(\cite{yano2}). In fact we prove that in any normal almost
bi-paracontact manifold the $1$-dimensional foliation ${\mathcal
F}_{\xi}$ defined by the Reeb vector field is transversely
anti-hypercomplex or complex-product, i.e.  the almost
bi-paracontact structure $(\phi_{1},\phi_{2},\phi_{3})$ is
projectable to a local anti-hypercomplex structure on the leaf
space.

We further investigate the curvature properties of this connection,
proving that, under the assumption of normality, its curvature
tensor field $R^{c}$ is of type $(1,1)$ with respect to $\phi_1$,
$\phi_2$, $\phi_3$, i.e.
$R^{c}(\phi_{1}X,\phi_{1}Y)=R^{c}(\phi_{2}X,\phi_{2}Y)=-R^{c}(\phi_{3}X,\phi_{3}Y)=-R^{c}(X,Y)$
for all $X,Y\in\Gamma(TM)$.

In the second part of the paper we apply our general results on
almost bi-paracontact structures to the theory of contact metric
$(\kappa,\mu)$-spaces. First, we study the bi-Legendrian structure
$({\mathcal D}_{\phi h}(\lambda),{\mathcal D}_{\phi h}(-\lambda))$.
We prove that the Legendre foliations ${\mathcal D}_{\phi
h}(\lambda)$ and ${\mathcal D}_{\phi h}(-\lambda)$ are either
non-degenerate or flat, according to the Pang's classification of
Legendre foliations (cf. \cite{pang}). In particular, ${\mathcal
D}_{\phi h}(\lambda)$ and ${\mathcal D}_{\phi h}(-\lambda)$ are
positive definite if and only if $I_M>0$, negative definite if and
only if $I_M<0$, flat if and only if $I_M=0$, where
\begin{equation*}
I_{M}:=\frac{1-\frac{\mu}{2}}{\sqrt{1-\kappa}}
\end{equation*}
is the invariant introduced by Boeckx for classifying contact metric
$(\kappa,\mu)$-structures. This provides a new geometrical
interpretation of such invariant in terms of Legendre foliations
(another one was given in \cite{Mino-toap1}).

Then we consider the almost bi-paracontact structure
$(\phi_{1},\phi_{2},\phi_{3})$ defined by \eqref{kappamu} and prove
that the semi-Riemannian metrics $g_1$ and $g_2$, given by
\begin{equation*}
g_{1}:=d\eta(\cdot,\phi_{1}\cdot)+\eta\otimes\eta, \ \ \
g_{2}:=d\eta(\cdot,\phi_{2}\cdot)+\eta\otimes\eta,
\end{equation*}
define two associated paracontact metrics satisfying
\begin{equation*}
R^{g_\alpha}(X,Y)\xi=\kappa_{\alpha}(\eta(Y)X-\eta(X)Y)+\mu_{\alpha}(\eta(Y)h_{\alpha}X-\eta(X)h_{\alpha}Y)
\end{equation*}
where
\begin{gather*}
\kappa_{1}=\left(1-\frac{\mu}{2}\right)^2-1, \ \ \
\mu_{1}=2(1-\sqrt{1-\kappa}),\\
\kappa_{2}=\kappa-2+\left(1-\frac{\mu}{2}\right)^2, \ \ \ \mu_{2}=2.
\end{gather*}
Mreover, $I_{M}=0$ if and only if $(\phi_{1},\xi,\eta,g_{1})$ is
para-Sasakian. \ Furthermore, we prove that any contact metric
$(\kappa,\mu)$-space such that $I_{M}\neq\pm 1$ admits a
supplementary non-normal almost bi-paracontact structure, although
one of the two paracontact structures is normal (cf. Theorem
\ref{main4}). In this way we obtain a class of examples of strictly
non-normal, integrable almost bi-paracontact structures.

Finally, we deal with the following question, which generalizes the
well-known problem of finding conditions ensuring the existence of
Sasakian structures compatible with a given contact form: let
$(M,\eta)$ be a contact manifold; then does $(M,\eta)$ admit a
compatible contact metric $(\kappa,\mu)$-structure? As a matter of
fact, the answer to this question involves the standard almost
bi-paracontact structure \eqref{kappamu} of  contact metric
$(\kappa,\mu)$-spaces. In particular, using the properties of the
canonical connection $\nabla^{c}$, we find necessary conditions for
a contact manifold $(M,\eta)$ endowed with an almost bi-paracontact
structure to admit a compatible contact metric
$(\kappa,\mu)$-structure (cf. Theorem \ref{main3}).

\section{Preliminaries}

\subsection{Almost contact and paracontact
structures}\label{preliminari1}

A \emph{contact manifold} is a $(2n+1)$-dimensio-nal smooth manifold
$M$ which carries a $1$-form $\eta$, called \emph{contact form},
satisfying the condition $\eta\wedge\left(d\eta\right)^n\neq 0$
everywhere on $M$. It is well known that given $\eta$ there exists a
unique vector field $\xi$, called \emph{Reeb vector field}, such
that
\begin{equation}\label{contatto4}
i_{\xi}\eta=1, \ \ \ i_{\xi}d\eta=0.
\end{equation}
From \eqref{contatto4} it follows that ${\mathcal L}_{\xi}d\eta=0$,
i.e. the $1$-dimensional foliation ${\mathcal F}_{\xi}$ defined by
the Reeb vector field is transversely symplectic. \ In the sequel we
will denote by $\mathcal D$ the $2n$-dimensional distribution
defined by $\ker\left(\eta\right)$, called the \emph{contact
distribution}. It is easy to see that the Reeb vector field is an
infinitesimal automorphism with respect to the contact distribution
and  the tangent bundle of $M$ splits as the direct sum $TM=\mathcal
D\oplus\mathbb{R}\xi$.

Given a contact manifold $(M,\eta)$ one can consider two different
geometric structures associated with the contact form $\eta$, namely
a ``contact metric structure'' and a ``paracontact metric
structure''.

An \emph{almost contact structure} on a $(2n+1)$-dimensional smooth
manifold $M$ is nothing but a triplet $(\phi,\xi,\eta)$, where
$\phi$ is a tensor field of type $(1,1)$, $\eta$ a $1$-form and
$\xi$ a vector field on $M$ satisfying the following conditions
\begin{equation}\label{almostcontact}
\phi^{2}=-I+\eta\otimes\xi, \ \ \eta(\xi)=1,
\end{equation}
where $I$ is the identity mapping. From \eqref{almostcontact} it
follows that $\phi\xi=0$, $\eta\circ\phi=0$ and the $(1,1)$-tensor
field $\phi$ has constant rank $2n$ (\cite{Blair-02}). Given an
almost contact manifold $(M,\phi,\xi,\eta)$ one can define an almost
complex structure $J$ on the product $M\times\mathbb{R}$ by setting
$J\left(X,f\frac{d}{dt}\right)=\left(\phi
X-f\xi,\eta\left(X\right)\frac{d}{dt}\right)$ for any
$X\in\Gamma\left(TM\right)$ and $f\in
C^{\infty}\left(M\times\mathbb{R}\right)$. Then the almost contact
manifold is said to be \emph{normal} if the almost complex structure
$J$ is integrable. The computation of the Nijenhuis tensor of $J$
gives rise to the four tensors defined by
\begin{align}
&N^{(1)}_{\phi}\left(X,Y\right)=\left[\phi,\phi\right]\left(X,Y\right)+2d\eta\left(X,Y\right)\xi,\label{n1}\\
&N^{(2)}_{\phi}\left(X,Y\right)=\left({\mathcal L}_{\phi
X}\eta\right)\left(Y\right)-\left({\mathcal L}_{\phi
Y}\eta\right)\left(X\right),\label{n2}\\
&N^{(3)}_{\phi}\left(X\right)=\left({\mathcal L}_{\xi}\phi\right)X,\label{n3}\\
&N^{(4)}\left(X\right)=\left({\mathcal
L}_{\xi}\eta\right)\left(X\right),\label{n4}
\end{align}
where $\left[\phi,\phi\right]$ is the Nijenhuis tensor of $\phi$,
defined by
\begin{equation*}
[\phi,\phi](X,Y):=\phi^2[X,Y]+[\phi X,\phi Y]-\phi[\phi
X,Y]-\phi[X,\phi Y],
\end{equation*}
and ${\mathcal L}_X$ denotes the Lie derivative with respect to the
vector field $X$. One finds that the structure
$\left(\phi,\xi,\eta\right)$ is normal if and only if $N^{(1)}$
vanishes identically; in particular, if $N^{(1)}_{\phi}=0$ then also
the other tensors $N^{(2)}_{\phi}$, $N^{(3)}_{\phi}$ and
$N^{(4)}_{\phi}$ vanish (cf. \cite{sasaki2}). \ By a long but
straightforward computation one can prove the following lemma which
will turn out very useful in the sequel.

\begin{lemma}
In any almost contact manifold $(M,\phi,\xi,\eta)$  for any
$X,Y\in\Gamma(TM)$,
\begin{equation}\label{lemma2}
\phi{N^{(1)}_{\phi}(X,Y)}+N^{(1)}_{\phi}(\phi X,Y) =
N^{(2)}_{\phi}(X,Y)\xi + \eta(X)N_{\phi}^{(3)}(Y).
\end{equation}
\end{lemma}

Any almost contact manifold $\left(M,\phi,\xi,\eta\right)$ admits a
\emph{compatible metric}, i.e. a Riemannian metric $g$ satisfying
\begin{equation}\label{metric1}
g\left(\phi X,\phi
Y\right)=g\left(X,Y\right)-\eta\left(X\right)\eta\left(Y\right)
\end{equation}
for all $X,Y\in\Gamma\left(TM\right)$. The manifold $M$ is  said to
be an \emph{almost contact metric manifold} with structure
$\left(\phi,\xi,\eta,g\right)$. From \eqref{metric1} it follows
immediately that $\eta=g(\cdot,\xi)$ and
$g(\cdot,\phi\cdot)=-g(\phi\cdot,\cdot)$. Then one defines the
$2$-form $\Phi$ on $M$ by $\Phi\left(X,Y\right)=g\left(X,\phi
Y\right)$, called the \emph{fundamental $2$-form} of the almost
contact metric manifold. If $\Phi=d\eta$ then $\eta$ becomes a
contact form, with $\xi$ its corresponding Reeb vector field, and
$(M,\phi,\xi,\eta,g)$ is \emph{called contact metric manifold}.

In a contact metric manifold one has
\begin{equation}\label{acca2}
\nabla^{g}\xi=-\phi -\phi h
\end{equation}
\begin{equation}\label{contatto2}
N^{(1)}_{\phi}(X,Y)=(\nabla^{g}_{\phi X}\phi)Y-(\nabla^{g}_{\phi
Y}\phi)X+(\nabla^{g}_{X}\phi)\phi Y-(\nabla^{g}_{Y}\phi)\phi
X-\eta(Y)\nabla^{g}_{X}\xi+\eta(X)\nabla^{g}_{Y}\xi
\end{equation}
where $\nabla^{g}$ is the Levi Civita connection of $(M,g)$ and
$h:=\frac{1}{2}N^{(3)}_{\phi}$. The tensor field $h$ is symmetric
with respect to $g$ and vanishes identically if and only if the Reeb
vector field is Killing, and in this case the contact metric
manifold is said to be \emph{K-contact}.
A  normal contact metric manifold is called \emph{Sasakian
manifold}. Any Sasakian manifold is also $K$-contact and the
converse holds only in dimension $3$.
A contact metric manifold is said to be \emph{integrable} if and
only if the following condition is fulfilled
\begin{equation}\label{integrabile1}
({\nabla}^{g}_X\phi)Y=g\left(X+{h}X,Y\right)\xi-\eta\left(Y\right)\left(X+{h}X\right).
\end{equation}
Any Sasakian manifold satisfies such condition. By replacing
\eqref{integrabile1} and \eqref{acca2} in \eqref{contatto2} one can
prove the following

\begin{proposition}
In an integrable contact metric manifold
\begin{equation}\label{contatto3}
N^{(1)}_{\phi}(X,Y)=2\bigl(\eta(Y)\phi h X-\eta(X)\phi h Y\bigr).
\end{equation}
\end{proposition}

\begin{corollary}
Any integrable $K$-contact manifold is Sasakian.
\end{corollary}

On the other hand on a contact manifold $(M,\eta)$ one can consider
also compatible paracontact metric structures. We recall (cf.
\cite{kaneyuki1}) that an \emph{almost paracontact structure} on a
$(2n+1)$-dimensional smooth manifold $M$ is given by a
$(1,1)$-tensor field $\tilde\phi$, a vector field $\xi$ and a
$1$-form $\eta$ satisfying the following conditions
\begin{enumerate}
  \item[(i)] $\eta(\xi)=1$, \ $\tilde\phi^2=I-\eta\otimes\xi$,
  \item[(ii)] the tensor field $\tilde\phi$ induces an almost paracomplex
  structure on each fibre on ${\mathcal D}=\ker(\eta)$.
\end{enumerate}
Recall that an almost paracomplex structure on a $2n$-dimensional
smooth manifold is a tensor field $\tilde{J}$ of type $(1,1)$ such
that $\tilde{J}\neq I$, $\tilde{J}^2=I$ and the eigendistributions
$T^+, T^-$ corresponding to the eigenvalues $1, -1$ of $\tilde{J}$,
respectively, have dimension $n$.

As an immediate consequence of the definition one has that
$\tilde\phi\xi=0$, $\eta\circ\tilde\phi=0$ and the field of
endomorphisms $\tilde\phi$ has constant rank $2n$. As for the almost
contact case, one can consider the almost paracomplex structure on
$M\times\mathbb{R}$ defined by
$\tilde{J}\bigl(X,f\frac{d}{dt}\bigr)=\bigl(\tilde{\phi}X+f\xi,\eta(X)\frac{d}{dt}\bigr)$,
where $X$ is a vector field on $M$ and $f$ a $C^{\infty}$ function
on $M\times\mathbb{R}$. By definition, if $\tilde{J}$ is integrable,
the almost paracontact structure $(\tilde{\phi},\xi,\eta)$ is said
to be \emph{normal}. The computation of $\tilde{J}$ in terms of the
tensors of the almost paracontact structure leads us to define four
tensors
\begin{align}
&N^{(1)}_{\tilde\phi}(X,Y)=[\tilde\phi,\tilde\phi](X,Y)-2d\eta(X,Y)\xi,\label{n1p}\\
&N^{(2)}_{\tilde\phi}(X,Y)=({\mathcal
L}_{\tilde{\phi}X}\eta)(Y)-({\mathcal L}_{\tilde{\phi}Y}\eta)(X),\label{n2p}\\
&N^{(3)}_{\tilde\phi}(X)=({\mathcal L}_{\xi}\tilde{\phi})X,\label{n3p}\\
&N^{(4)}(X)=({\mathcal L}_{\xi}\eta)(X),\label{n4p}
\end{align}
The almost paracontact structure is then normal if and only if these
four tensors vanish. However, as it is shown in \cite{zamkovoy}, the
vanishing of $N^{(1)}_{\tilde\phi}$ implies the vanishing of the
remaining tensors.

Any almost paracontact manifold admits a semi-Riemannian metric
$\tilde g$ such that
\begin{equation}\label{compatibile}
\tilde g(\tilde\phi X,\tilde\phi Y)=-\tilde g(X,Y)+\eta(X)\eta(Y)
\end{equation}
for all  $X,Y\in\Gamma(TM)$. Then $(M,\tilde\phi,\xi,\eta,\tilde g)$
is called an \emph{almost paracontact metric manifold}. Notice that
any such a semi-Riemannian metric is necessarily of signature
$(n+1,n)$. Moreover, as in the almost contact case, from
\eqref{compatibile} it follows easily that $\eta=g(\cdot,\xi)$ and
$\tilde{g}(\cdot,\tilde\phi\cdot)=-\tilde{g}(\tilde\phi\cdot,\cdot)$.
Hence one defines the \emph{fundamental $2$-form} of the almost
paracontact metric manifold by
$\tilde\Phi(X,Y)=\tilde{g}(X,\tilde\phi Y)$. If $d\eta=\tilde\Phi$,
$\eta$ becomes a contact form and $(M,\tilde\phi,\xi,\eta,\tilde g)$
is said to be a \emph{paracontact metric manifold}.

On a paracontact metric manifold $(M,\tilde\phi,\xi,\eta,\tilde g)$
one has
\begin{equation}\label{acca3}
\nabla^{\tilde{g}}\xi=-\tilde\phi+\tilde\phi\tilde h
\end{equation}
\begin{equation}\label{paracontatto2}
N^{(1)}_{\tilde\phi}(X,Y)=(\nabla^{\tilde g}_{\tilde\phi
X}\tilde\phi)Y-(\nabla^{\tilde g}_{\tilde\phi
Y}\tilde\phi)X+(\nabla^{\tilde g}_{X}\tilde\phi)\tilde\phi
Y-(\nabla^{\tilde g}_{Y}\tilde\phi)\tilde\phi
X+\eta(Y)\nabla^{\tilde g}_{X}\xi-\eta(X)\nabla^{\tilde g}_{Y}\xi
\end{equation}
where $\tilde{h}:=\frac{1}{2}N_{\tilde\phi}^{(3)}$. One proves (see
\cite{zamkovoy}) that $\tilde h$ is symmetric with respect to
$\tilde{g}$  and $\tilde h$ vanishes identically if and only if
$\xi$ is a Killing vector field and in such case
$(M,\tilde\phi,\xi,\eta,\tilde g)$ is called a \emph{$K$-paracontact
manifold}. By using \eqref{acca3} one can prove (cf.
\cite{Mino-toap2}) the formula
\begin{equation}\label{formulacurvatura1}
R^{\tilde{g}}(X,Y)\xi=-(\nabla^{\tilde{g}}_{X}\tilde\phi)Y+(\nabla^{\tilde{g}}_{Y}\tilde\phi)X+(\nabla^{\tilde{g}}_{X}\tilde\phi)\tilde
h Y+\tilde\phi((\nabla^{\tilde{g}}_{X}\tilde
h)Y)-(\nabla^{\tilde{g}}_{Y}\tilde\phi)\tilde h
X-\tilde\phi((\nabla^{\tilde{g}}_{Y}\tilde h)X).
\end{equation}
A normal paracontact metric manifold is said to be a
\emph{para-Sasakian manifold}. Also in this context the
para-Sasakian condition implies the \emph{K}-paracontact condition
and the converse holds in dimension $3$. In terms of the covariant
derivative of $\tilde\phi$ the para-Sasakian condition may be
expressed by
\begin{equation}\label{condizioneparasasaki}
(\nabla^{\tilde{g}}_{X}\tilde\phi)Y=-\tilde g(X,Y)\xi+\eta(Y)X.
\end{equation}

In any paracontact metric manifold Zamkovoy introduced a canonical
connection which  plays the same role in paracontact geometry of the
generalized Tanaka-Webster connection (\cite{tanno}) in a contact
metric manifold. In fact the following result holds.

\begin{theorem}[\cite{zamkovoy}]\label{paratanaka}
On a paracontact metric manifold there exists a unique connection
${\nabla}^{pc}$, called the \emph{canonical paracontact connection},
satisfying the following properties:
\begin{enumerate}
  \item[(i)] ${\nabla}^{pc}\eta=0$, ${\nabla}^{pc}\xi=0$,
  ${\nabla}^{pc}\tilde{g}=0$,
  \item[(ii)]
  $(\nabla^{pc}_{X}\tilde\phi)Y=(\nabla^{\tilde{g}}_{X}\tilde\phi)Y-\eta(Y)(X-\tilde{h}X)+\tilde{g}(X-\tilde{h}X,Y)\xi$,
  \item[(iii)] $T^{pc}(\xi,\tilde\phi Y)=-\tilde\phi \tilde T^{pc}(\xi,Y)$,
  \item[(iv)] $T^{pc}(X,Y)=2d\eta(X,Y)\xi$ \ for all $X,Y\in\Gamma({\mathcal D})$.
\end{enumerate}
The explicit expression of this connection is the following
\begin{equation}\label{paradefinition}
\nabla^{pc}_{X}Y=\nabla^{\tilde{g}}_{X}Y+\eta(X)\tilde\phi
Y+\eta(Y)(\tilde\phi X-\tilde\phi \tilde{h}X)+\tilde
g(X-\tilde{h}X,\tilde\phi Y)\xi.
\end{equation}
Moreover, the torsion tensor field is given by
\begin{equation}\label{paratorsion}
T^{pc}(X,Y)=\eta(X)\tilde\phi\tilde{h}Y - \eta(Y)\tilde\phi
\tilde{h} X + 2\tilde{g}(X,\tilde\phi Y)\xi.
\end{equation}
\end{theorem}

If the paracontact metric connection preserves the structure tensor
$\tilde\phi$, that is the Levi Civita connection satisfies
\begin{equation}\label{integrabile2}
(\nabla^{\tilde{g}}_{X}\tilde\phi)Y=\eta(Y)(X-\tilde{h}X)-\tilde{g}(X-\tilde{h}X,Y)\xi
\end{equation}
for any $X,Y\in\Gamma(TM)$, then the paracontact metric structure
$(\tilde\phi,\xi,\eta,\tilde{g})$ is said to be \emph{integrable}.
This is the case, in particular, when the eigendistributions
$T^{\pm}$ of $\tilde\phi$ associated to the eigenvalues $\pm 1$ are
involutive. Moreover, \eqref{integrabile2} and
\eqref{condizioneparasasaki} it follows that any para-Sasakian
manifold is integrable. By replacing \eqref{integrabile2} and
\eqref{acca3} in \eqref{paracontatto2} one can straightforwardly
prove the following proposition.

\begin{proposition}
In a integrable paracontact metric manifold
\begin{equation}\label{paracontatto3}
N^{(1)}_{\tilde\phi}(X,Y)=2\bigl(\eta(Y)\tilde\phi\tilde{h}X-\eta(X)\tilde\phi\tilde{h}Y\bigr).
\end{equation}
\end{proposition}

\begin{corollary}\label{paracontatto4}
Any integrable $K$-paracontact manifold is  para-Sasakian.
\end{corollary}

\subsection{Bi-Legendrian manifolds}\label{preliminari}
Let $(M,\eta)$ be a $(2n+1)$-dimensional contact manifold. It is
well-known that the contact condition $\eta\wedge(d\eta)^n\neq 0$
geometrically means that the contact distribution $\mathcal D$ is as
far as possible from being integrable. In fact one can prove that
the maximal dimension of an involutive subbundle of $\mathcal D$ is
$n$. Such $n$-dimensional integrable distributions are called
\emph{Legendre foliations} of $(M,\eta)$. More  generally a
\emph{Legendre distribution} on a contact manifold $(M,\eta)$ is an
$n$-dimensional subbundle $L$ of the contact distribution not
necessarily integrable but verifying the weaker condition that
$d\eta\left(X,X'\right)=0$ for all $X,X'\in\Gamma\left(L\right)$. \
The theory of Legendre foliations has been extensively investigated
in recent years from various points of views. In particular  Pang
(\cite{pang}) provided a classification of Legendre foliations by
using
 a bilinear symmetric form $\Pi_{\mathcal F}$ on the tangent bundle
of the foliation ${\mathcal F}$, defined by
\begin{equation}\label{invariante3}
\Pi_{\mathcal F}\left(X,X'\right)=-\left({\mathcal L}_{X}{\mathcal
L}_{X'}\eta\right)\left(\xi\right)=2d\eta([\xi,X],X').
\end{equation}
He called a Legendre foliation \emph{positive (negative) definite},
\emph{non-degenerate}, \emph{degenerate} or \emph{flat} according to
the circumstance that the bilinear form $\Pi_{\mathcal F}$ is
positive (negative) definite, non-degenerate, degenerate or vanishes
identically, respectively. By \eqref{invariante3} it follows that
$\mathcal F$ is flat if and only if $\xi$ is ``foliate'', i.e.
$[\xi,X]\in\Gamma(T{\mathcal F})$ for any $X\in\Gamma(T{\mathcal
F})$.

If $(M,\eta)$  is endowed with  two  transversal  Legendre
distributions $L_1$ and $L_2$,  we  say  that $(M,\eta,L_1,L_2)$ is
an \emph{almost bi-Legendrian manifold}. Thus, in particular, the
tangent bundle of $M$ splits up as the direct sum $TM=L_1\oplus
L_2\oplus\mathbb{R}\xi$. When both $L_1$ and $L_2$ are integrable we
refer to a \emph{bi-Legendrian manifold}. An (almost) bi-Legendrian
manifold is said to be flat, degenerate or non-degenerate if and
only if both the Legendre distributions are flat, degenerate or
non-degenerate, respectively.   Any contact manifold $(M,\eta)$
endowed with a Legendre distribution $L$ admits a canonical almost
bi-Legendrian structure. Indeed let $(\phi,\xi,\eta,g)$ be a
compatible contact metric structure. Then  the relation $d\eta(\phi
X,\phi Y)=\Phi(\phi X,\phi Y)=d\eta(X,Y)$  easily implies that
$Q:=\phi L$ is a Legendre distribution on $M$ which is
$g$-orthogonal to $L$. $Q$ is usually referred as the
\emph{conjugate Legendre distribution} of $L$ and in general is not
involutive, even if $L$ is. \ In \cite{Mino-05} the existence of  a
canonical connection on an almost bi-Legendrian manifold has been
proven:

\begin{theorem}[\cite{Mino-05}]\label{biconnection}
Let $(M,\eta,L_1,L_2)$ be an almost bi-Legendrian manifold. There
exists a unique linear connection ${\nabla}^{bl}$ called  the
\emph{bi-Legendrian connection}, satisfying the following
properties:
\begin{enumerate}
  \item[(i)] ${\nabla}^{bl} L_1\subset L_1$, \ ${\nabla}^{bl} L_2\subset L_2$,
  \item[(ii)]  ${\nabla}^{bl}\xi=0$, \ ${\nabla}^{bl} d\eta=0$,
  \item[(iii)] ${T}^{bl}\left(X,Y\right)=2d\eta\left(X,Y\right){\xi}$ \ for all
  $X\in\Gamma(L_1)$, $Y\in\Gamma(L_2)$,\\
${T}^{bl}\left(X,\xi\right)=[\xi,X_{L_1}]_{L_2}+[\xi,X_{L_2}]_{L_1}$
  \ for all   $X\in\Gamma\left(TM\right)$,
\end{enumerate}
where $X_{L_1}$ and $X_{L_2}$ the projections of $X$ onto the
subbundles $L_1$ and $L_2$ of $TM$, respectively. Furthermore, the
torsion tensor field $T^{bl}$ of $\nabla^{bl}$ is explicitly given
by
\begin{align}\label{torsione}
T^{bl}(X,Y)&=-[X_{L_1},Y_{L_1}]_{L_2\oplus\mathbb{R}\xi}-[X_{L_2},Y_{L_2}]_{L_1\oplus\mathbb{R}\xi}+2d\eta(X,Y)\xi\nonumber\\
&\quad+\eta(Y)\left([\xi,X_{L_1}]_{L_2}+[\xi,X_{L_2}]_{L_1}\right)-\eta(X)\left([\xi,Y_{L_1}]_{L_2}+[\xi,Y_{L_2}]_{L_1}\right).
\end{align}
\end{theorem}

%

In \cite{Mino-08} the interplays between paracontact geometry and
the theory of bi-Legendrian structures have been studied. More
precisely it has been proven the existence of a biunivocal
correspondence $\Psi:{\mathcal{B}}\longrightarrow{\mathcal{P}}$
 between the set ${\mathcal{B}}$ of  almost bi-Legendrian structures
 and the set $\mathcal{P}$ of paracontact metric structures  on the same contact manifold
$(M,\eta)$. This bijection maps bi-Legendrian structures onto
integrable paracontact structures, flat almost bi-Legendrian
structures onto \emph{K}-paracontact structures and flat
bi-Legendrian structures onto para-Sasakian structures. For the
convenience of the reader we recall more explicitly how the above
biunivocal correspondence is defined. If $(L_1,L_2)$ is an almost
bi-Legendrian structure on $(M,\eta)$, the corresponding paracontact
metric structure $(\tilde\phi,\xi,\eta,\tilde g)=\Psi(L_1,L_2)$ is
given by
\begin{equation}\label{biunivoca}
\tilde\phi|_{L_1}=I, \ \tilde\phi|_{L_2}=-I, \ \tilde\phi\xi=0, \ \
\tilde g:=d\eta(\cdot,\tilde\phi \cdot)+\eta\otimes\eta.
\end{equation}
Moreover, the relationship between the bi-Legendrian and the
canonical paracontact connections has been investigated, proving
that in the integrable case they in fact coincide:

\begin{theorem}[\cite{Mino-08}]\label{connection}
Let  $(M,\eta,L_1,L_2)$  be  an  almost  bi-Legendrian  manifold and
 let $(\tilde\phi,\xi,\eta,\tilde g)=\Psi(L_1,L_2)$ be the paracontact metric
structure induced on $M$ by \eqref{biunivoca}. Let $\nabla^{bl}$ and
$\nabla^{pc}$ be the corresponding bi-Legendrian and canonical
paracontact connections. Then
\begin{enumerate}
  \item[(a)] $\nabla^{bl}\tilde\phi=0$, $\nabla^{bl}\tilde g=0$,
  \item[(b)] the bi-Legendrian and the canonical paracontact connections coincide if and only if
  the induced paracontact metric structure is integrable.
\end{enumerate}
\end{theorem}

\section{Almost bi-paracontact structures on contact
manifolds}\label{biparacontact}

\begin{definition}
Let $(M,\eta)$ be a contact manifold. An \emph{almost bi-paracontact
structure} on $(M,\eta)$ is a triplet $(\phi_1,\phi_2,\phi_3)$ where
$\phi_3$ is an almost contact structure compatible with $\eta$, and
$\phi_1$, $\phi_2$ are two anti-commuting tensors on $M$ such that
$\phi_{1}^{2}=\phi_{2}^{2}=I-\eta\otimes\xi$ and
$\phi_{1}\phi_{2}=\phi_3$.
\end{definition}

The manifold $M$ endowed with such a geometrical structure is called
an \emph{almost bi-paracontact manifold}. From the definition it
easily follows that $\phi_{1}\phi_{3}=-\phi_{3}\phi_{1}=\phi_{2}$
and $\phi_{3}\phi_{2}=-\phi_{2}\phi_{3}=\phi_{1}$.


For each $\alpha\in\left\{1,2,3\right\}$ let
${\mathcal{D}}^{+}_{\alpha}$ and ${\mathcal{D}}^{-}_{\alpha}$ denote
the eigendistributions of $\phi_\alpha$ corresponding, respectively,
to the eigenvalues $1$ and $-1$.
Notice that, as $\phi_\alpha\xi=0$, ${\mathcal{D}}^{+}_{\alpha}$ and
${\mathcal{D}}^{-}_{\alpha}$ are in fact subbundles of the contact
distribution. In the following proposition we collect  some
properties of those distributions.

\begin{proposition}\label{proprieta1}
Let $(M,\eta,\phi_1,\phi_2,\phi_3)$ be an almost bi-paracontact
manifold. Then
\begin{itemize}
  \item[1.] $\phi_1({\mathcal D}_{2}^{+})={\mathcal D}_{2}^{-}$, $\phi_1({\mathcal D}_{2}^{-})={\mathcal
  D}_{2}^{+}$,
  \item[2.] $\phi_2({\mathcal D}_{1}^{+})={\mathcal D}_{1}^{-}$, $\phi_2({\mathcal D}_{1}^{-})={\mathcal
  D}_{1}^{+}$,
  \item[3.] $\phi_3({\mathcal D}_{\alpha}^{+})={\mathcal D}_{\alpha}^{-}$, $\phi_3({\mathcal D}_{\alpha}^{-})={\mathcal
  D}_{\alpha}^{+}$ for each $\alpha\in\left\{1,2\right\}$,
  \item[4.] $\phi_1:{\mathcal D}_{2}^{+}\longrightarrow{\mathcal
  D}_{2}^{-}$ and $\phi_2:{\mathcal D}_{1}^{+}\longrightarrow{\mathcal
  D}_{1}^{-}$ are isomorphisms,
  \item[5.] the tangent bundle of $M$ splits up as the direct sum
  $TM={\mathcal D}_{\alpha}^{+}\oplus{\mathcal
  D}_{\alpha}^{-}\oplus\mathbb{R}\xi={\mathcal D}_{\alpha}^{\pm}\oplus{\mathcal
  D}_{\beta}^{\pm}\oplus\mathbb{R}\xi$ for all
$\alpha,\beta\in\left\{1,2\right\}$, $\alpha\neq\beta$,
  \item[6.] $\dim({\mathcal D}_{1}^{+})=\dim({\mathcal D}_{1}^{-})=\dim({\mathcal D}_{2}^{+})=\dim({\mathcal
  D}_{2}^{-})=n$. In particular, $\phi_1$ and $\phi_2$ are almost
  paracontact structures.
\end{itemize}
\end{proposition}
\begin{proof}
For any $X\in\Gamma({\mathcal D}_{2}^{+})$ one has $\phi_2\phi_1
X=-\phi_1\phi_2 X=-\phi_1 X$, so that $\phi_1({\mathcal
D}_{2}^{+})\subset{\mathcal D}_{2}^{-}$. On the other hand, let $Y$
be a vector field tangent to ${\mathcal D}_{2}^{-}$ and set
$X:=\phi_1 Y$. Then $\phi_1 X=\phi_{1}^{2}Y=Y$, so that it remains
only to prove that $X\in\Gamma({\mathcal D}_{2}^{+})$. Indeed,
$\phi_2 X=\phi_2\phi_1 Y=-\phi_1\phi_2 Y=\phi_1 Y=X$. Thus
$\phi_1({\mathcal D}_{2}^{+})={\mathcal D}_{2}^{-}$ and analogously
one can prove that $\phi_1({\mathcal D}_{2}^{-})={\mathcal
D}_{2}^{+}$. In a similar way one proves the other identities, as
well as the fourth property.  \ In order to prove the fifth property
it is enough to show that ${\mathcal D}={\mathcal
D}_{\alpha}^{+}\oplus{\mathcal D}_{\alpha}^{-}={\mathcal
D}_{\alpha}^{\pm}\oplus{\mathcal D}_{\beta}^{\pm}$ for all
$\alpha,\beta\in\left\{1,2\right\}$. Let us consider the case
$\alpha=1$. Then we can decompose every $X\in\Gamma({\mathcal D})$
as $X=\frac{1}{2}\left(X-\phi_{1}X\right)+\frac{1}{2}\left(X+\phi_1
X\right)$. An immediate computation shows that
$\frac{1}{2}\left(X-\phi_{1}X\right)\in{\mathcal D}^{-}_{1}$ and
$\frac{1}{2}\left(X+\phi_{1}X\right)\in{\mathcal D}^{+}_{1}$. Next,
if $X\in{\mathcal D}^{+}_{1}\cap{\mathcal D}^{-}_{1}$ then $\phi_1
X=X=-\phi_1 X$, from which it follows that $\phi_1^2 X = -\phi_1^2
X$, hence $X=0$. Thus ${\mathcal D}={\mathcal
D}^{+}_{1}\oplus{\mathcal D}^{-}_{1}$. In a similar way one can
prove that ${\mathcal D}={\mathcal D}^{+}_{2}\oplus{\mathcal
D}^{-}_{2}$. Now we prove the identity ${\mathcal D}={\mathcal
D}_{1}^{+}\oplus{\mathcal D}_{2}^{+}$. If $X\in{\mathcal
D}_{1}^{+}\cap{\mathcal D}_{2}^{+}$ then $\phi_1 X = X = \phi_2 X$,
hence $X=\phi_1\phi_{2}X=\phi_3 X$ and this implies that $X=0$. On
the other hand, note that from 4, since ${\mathcal D}={\mathcal
D}^{+}_{\alpha}\oplus{\mathcal D}^{-}_{\alpha}$,
$\alpha\in\left\{1,2\right\}$,  it follows that, for each
$\alpha\in\left\{1,2\right\}$, $\dim({\mathcal
D}^{+}_{\alpha})=\dim({\mathcal D}^{-}_{\alpha})=n$. Hence
$\dim({\mathcal D}_{1}^{+}+{\mathcal D}_{2}^{+})=2n$ and we conclude
that ${\mathcal D}={\mathcal D}_{1}^{+}\oplus{\mathcal D}_{2}^{+}$.
The other identities can be proven similarly.
\end{proof}

\begin{proposition}\label{proprieta2}
In any almost bi-paracontact manifold one has ${\mathcal
D}_{1}^{\pm}=\{X+\phi_3 X | X\in{\mathcal D}_{2}^{\pm}\}$ and
${\mathcal D}_{2}^{\pm}=\{X+\phi_3 X | X\in{\mathcal
D}_{1}^{\mp}\}$.
\end{proposition}
\begin{proof}
We show that ${\mathcal D}_{1}^{+}=\{X+\phi_3 X | X\in{\mathcal
D}_{2}^{+}\}$ by proving the two inclusions. Let $Y\in{\mathcal
D}_{1}^{+}$. We have to prove the existence of $X\in{\mathcal
D}_{2}^{+}$ such that $Y=X+\phi_{3}X$. We put
$X:=\frac{1}{2}(Y-\phi_{3}Y)$. Firstly we verify that in fact
$X\in{\mathcal D}_{2}^{+}$. We have
$\phi_{2}X=\frac{1}{2}(\phi_{2}Y-\phi_{2}\phi_{3}Y)=\frac{1}{2}(\phi_{2}Y+\phi_{1}Y)=\frac{1}{2}(\phi_{2}Y+Y)=\frac{1}{2}(Y+\phi_{2}\phi_{1}Y)=\frac{1}{2}(Y-\phi_{1}\phi_{2}Y)=X$,
hence $X\in{\mathcal D}_{2}^{+}$. Next, one can easily check that
$Y=X+\phi_{3}X$. Conversely, let $X$ be a vector field belonging to
${\mathcal D}_{2}^{+}$. Then
$\phi_{1}(X+\phi_{3}X)=\phi_{1}X+\phi_{1}\phi_{3}X=\phi_{1}X+\phi_{2}X=\phi_{3}\phi_{2}X+X=\phi_{3}X+X$,
so that $X+\phi_{3}X\in{\mathcal D}_{1}^{+}$. In a similar manner
one can prove the other equality.
\end{proof}

\begin{example}
\emph{Consider $\mathbb{R}^{2n+1}$ with global coordinates
$\{x_{1},\ldots,x_{n},y_{1},\ldots,y_{n},z\}$ and the standard
contact form $\eta=dz-\sum_{i=1}^{n}y_{i}dx_{i}$. Put, for each
$i\in\left\{1,\ldots,n\right\}$, $X_{i}:=\frac{\partial}{\partial
y_i}$ and $Y_{i}:=\frac{\partial}{\partial
x_i}+y_{i}\frac{\partial}{\partial z}$. Then the contact
distribution $\mathcal D$ is spanned by the vector fields
$X_{1},\ldots,X_{n},Y_{1},\ldots,Y_{n}$. We define three tensor
fields $\phi_1$, $\phi_2$, $\phi_3$ by setting}
\begin{align*}
\phi_{1}X_{i}&:=X_{i}, \ \ \phi_{1}Y_{i}:=-Y_i, \ \ \phi_{1}\xi:=0,\\
\phi_{2}X_{i}&:=-Y_{i}, \ \ \phi_{2}Y_{i}:=-X_i, \ \ \phi_{1}\xi:=0,
\end{align*}
\begin{align*}
\phi_{3}X_{i}&:=Y_{i}, \ \ \phi_{3}Y_{i}:=-X_i, \ \ \phi_{1}\xi:=0,
\end{align*}
\emph{for all $i\in\left\{1,\ldots,n\right\}$. Some straightforward
computations show that $(\phi_{1},\phi_{2},\phi_{3})$ defines a
bi-paracontact structure on the contact manifold
$(\mathbb{R}^{2n+1},\eta)$. In this case the canonical distributions
${\mathcal D}_{1}^{+}$, ${\mathcal D}_{1}^{-}$, ${\mathcal
D}_{2}^{+}$, ${\mathcal D}_{2}^{-}$ are given by}
\begin{align*}
{\mathcal
D}_{1}^{+}&=\emph{span}\left\{X_{1},\ldots,X_{n}\right\}, \ \ \
{\mathcal D}_{1}^{-}=\emph{span}\left\{Y_{1},\ldots,Y_{n}\right\},\\
{\mathcal
D}_{2}^{+}&=\emph{span}\left\{X_{1}-Y_{1},\ldots,X_{n}-Y_{n}\right\},
\ {\mathcal
D}_{2}^{-}=\emph{span}\left\{X_{1}+Y_{1},\ldots,X_{n}+Y_{n}\right\}.
\end{align*}
\end{example}

In order to find some more examples we prove the following
proposition.

\begin{proposition}\label{conjugate}
Let $(M,\phi,\xi,\eta,g)$ be a contact metric manifold endowed with
a Legendre distribution $L$. Then $M$ admits a canonical almost
bi-paracontact structure.
\end{proposition}
\begin{proof}
Let $Q$ be the conjugate Legendre distribution of $L$, i.e. the
Legendre distribution on $M$ defined by $Q:=\phi(L)$ (see $\S$
\ref{preliminari}).  We define the $(1,1)$-tensor field $\psi$ on
$M$ by setting $\psi|_{L}=I$, $\psi|_{Q}=-I$, $\psi\xi=0$. Then if
we put $\phi_{1}:=\phi\psi$, $\phi_{2}:=\psi$, $\phi_{3}:=\phi$, it
is not difficult to check that $(\phi_{1},\phi_{2},\phi_{3})$ is in
fact an almost bi-paracontact structure on $(M,\eta)$.
\end{proof}

As a consequence of Proposition \ref{conjugate} we obtain a
canonical almost bi-paracontact structure on the tangent sphere
bundle $T_{1}M$ of a Riemannian manifold $(M,g)$ and on any contact
metric $(\kappa,\mu)$-space (\cite{BKP-95}). We will examine
carefully this last example in the last section of the paper.

\begin{definition}
An almost bi-paracontact structure such that ${\mathcal
D}_{1}^{\pm}$ and ${\mathcal D}_{2}^{\pm}$ are Legendre
distributions is called a \emph{Legendrian bi-paracontact
structure}. If ${\mathcal D}_{1}^{\pm}$ and ${\mathcal D}_{2}^{\pm}$
define Legendre foliations of $(M,\eta)$ then the almost
bi-paracontact structure is called \emph{integrable}.
\end{definition}

We present some characterizations of the integrability of an almost
bi-paracontact manifold.


\begin{proposition}\label{proprieta4}
An almost bi-paracontact structure $(\phi_1,\phi_2,\phi_3)$ is
Legendrian if and only if for each $\alpha\in\left\{1,2\right\}$ the
tensor field $N_{\phi_{\alpha}}^{(2)}$ vanishes identically.
Furtheremore, in any Legendrian almost bi-paracontact structure
 also the tensor field $N_{\phi_{3}}^{(2)}$ vanishes identically. In
particular, one has,  for any $X,Y\in\Gamma(TM)$,
\begin{equation}\label{legendrian}
d\eta(\phi_{1}X,\phi_{1}Y)=d\eta(\phi_{2}X,\phi_{2}Y)=-d\eta(\phi_{3}X,\phi_{3}Y)=-d\eta(X,Y)
\end{equation}
\end{proposition}
\begin{proof}
First of all we have, for all $X\in\Gamma({\mathcal D})$,
$N^{(2)}_{\phi_{\alpha}}(\xi,X)=-\eta([\xi,\phi_{\alpha}X])=2d\eta(\xi,\phi_{\alpha}X)=0$
by \eqref{contatto4}. Next, in order to prove that
$N^{(2)}_{\phi_{\alpha}}$ vanishes on $\mathcal D$, we distinguish
the cases $X,Y\in\Gamma({\mathcal D}_{\alpha}^{+})$,
$X,Y\in\Gamma({\mathcal D}_{\alpha}^{-})$ and $X\in\Gamma({\mathcal
D}_{\alpha}^{\pm})$, $Y\in\Gamma({\mathcal D}_{\alpha}^{\mp})$. In
the first case we have $N^{(2)}_{\phi_{\alpha}}(X,Y)$ $=$
$\phi_{\alpha}X(\eta(Y))-\eta([\phi_{\alpha}X,Y])-\phi_{\alpha}Y(\eta(X))+\eta([\phi_{\alpha}Y,X])$
$=$ $2d\eta(\phi_{\alpha}X,Y)+2d\eta(X,\phi_{\alpha}Y)$ $=$
$4d\eta(X,Y)=0$, where the last equality is due to the fact that
${\mathcal D}_{\alpha}^{+}$ is a Legendre distribution. The case
$X,Y\in\Gamma({\mathcal D}_{\alpha}^{-})$ is similar. Next, for any
$X\in\Gamma({\mathcal D}_{\alpha}^{\pm})$, $Y\in\Gamma({\mathcal
D}_{\alpha}^{\mp})$, we have
$N^{(2)}_{\phi_{\alpha}}(X,Y)=-\eta([\phi_{\alpha}X,Y])+\eta([\phi_{\alpha}Y,X])=\mp\eta([X,Y])\pm\eta([X,Y])=0$.
Conversely, if $N^{(2)}_{\alpha}\equiv 0$ then, for any
$X,Y\in\Gamma({\mathcal D}_{\alpha}^{+})$,
$0=N^{(2)}_{\alpha}(X,Y)=2d\eta(\phi_{\alpha}X,Y)+2\eta(X,\phi_{\alpha}Y)=4d\eta(X,Y)$,
so that $d\eta(X,Y)=0$. Consequently, as, by Proposition
\ref{proprieta1}, ${\mathcal D}_{\alpha}^{+}$ is $n$-dimensional, it
is a Legendre distribution. In a similar way one can prove that also
${\mathcal D}_{\alpha}^{-}$ is a Legendre distribution. \ In order
to prove the second part of the proposition, notice that since
$N^{(2)}_{\phi_1}$ and $N^{(2)}_{\phi_2}$ vanish,  for each
$\alpha\in\left\{1,2\right\}$,
$d\eta(\phi_{\alpha}\cdot,\cdot)=-d\eta(\cdot,\phi_{\alpha}\cdot)$.
Now, for any $X,Y\in\Gamma(TM)$,
$d\eta(\phi_{3}X,Y)=d\eta(\phi_{1}\phi_{2}X,Y)=-d\eta(\phi_{2}X,\phi_{1}Y)=d\eta(X,\phi_{2}\phi_{1}Y)=-d\eta(X,\phi_{3}Y)$.
Hence,
$N_{\phi_{3}}^{(2)}(X,Y)=\phi_{3}X(\eta(Y))-\eta([\phi_{3}X,Y])-\phi_{3}Y(\eta(X))+\eta([\phi_{3}Y,X])=2d\eta(\phi_{3}X,Y)-2d\eta(\phi_{3}Y,X)=0$.
\end{proof}

\begin{proposition}\label{proprieta3}
An almost bi-paracontact structure $(\phi_1,\phi_2,\phi_3)$ is
Legendrian (respectively, integrable) if and only if, for each
$\alpha\in\left\{1,2\right\}$,
$N^{(1)}_{\phi_\alpha}(X,X')\in\Gamma({\mathcal D}^{\mp}_{\alpha})$
(respectively, $N^{(1)}_{\phi_\alpha}(X,X')=0$) for any
$X,X'\in\Gamma({\mathcal D}^{\pm}_{\alpha})$.
\end{proposition}
\begin{proof}
By \eqref{n1p} we have, for any $X,X'\in\Gamma({\mathcal
D}_{\alpha}^{+})$,
\begin{equation}\label{legendre1bis}
N^{(1)}_{\phi_\alpha}(X,X')=[X,X]+[X,X]-\phi_\alpha[X,X']-\phi_\alpha[X,X']=2[X,X']-2\phi_\alpha[X,X'].
\end{equation}
Hence, applying $\phi_\alpha$ one obtains
\begin{equation}\label{legendre1}
\phi_{\alpha}{N^{(1)}_{\phi_\alpha}(X,X')}=2\phi_{\alpha}[X,X']-2[X,X']+2\eta([X,X'])\xi=-N^{(1)}_{\phi_\alpha}(X,X')-4d\eta(X,X')\xi
\end{equation}
Then  \eqref{legendre1} implies that $d\eta(X,X')=0$ if and only if
$N^{(1)}_{\phi_\alpha}(X,X')\in\Gamma({\mathcal D}_{\alpha}^{-})$
and  \eqref{legendre1bis} that ${\mathcal D}_{\alpha}^{+}$ is
involutive if and only if $N^{(1)}_{\phi_\alpha}(X,X')=0$. Analogous
arguments work for ${\mathcal D}^{-}_{\alpha}$.
\end{proof}

\begin{corollary}\label{proprieta7}
An almost bi-paracontact structure $(\phi_1,\phi_2,\phi_3)$ is
integrable if and only if the tensor fields $N^{(1)}_{\phi_1}$,
$N^{(1)}_{\phi_2}$ vanish on the contact distribution $\mathcal D$.
Furthermore, in an integrable almost bi-paracontact manifold also
the tensor field  $N^{(1)}_{\phi_3}$  vanishes on $\mathcal D$.
\end{corollary}
\begin{proof}
The proof is trivial in one direction. Conversely, notice that, for
any $X\in\Gamma({\mathcal D}_{\alpha}^{+})$, $Y\in\Gamma({\mathcal
D}_{\alpha}^{-})$,
$N^{(1)}_{\phi_{\alpha}}(X,Y)=[X,Y]+[X,-X]-\phi_{\alpha}[X,Y]-\phi_{\alpha}[X,-Y]=0$.
Then by Proposition \ref{proprieta3} we have that $N^{(1)}_{\phi_1}$
and $N^{(1)}_{\phi_2}$ vanish on $\mathcal D$. Now for ending the
proof it remains to demonstrate that if $N^{(1)}_{\phi_1}$ and
$N^{(1)}_{\phi_2}$ vanish on $\mathcal D$ then also
$N^{(1)}_{\phi_3}$ vanishes on $\mathcal D$. Let $X$, $X'$ be
sections of ${\mathcal D}_{1}^{+}$. By Proposition \ref{proprieta1},
$\phi_{2}X$ and $\phi_{2}X'$ are sections of ${\mathcal D}_{1}^{-}$.
Then the integrability of ${\mathcal D}_{1}^{+}$ and ${\mathcal
D}_{1}^{-}$ yield
\begin{align}
0&=\phi_{1}N^{(1)}_{\phi_2}(X,X')\nonumber\\
&=\phi_{1}[X,X']+\phi_{1}[\phi_{2}X,\phi_{2}X']-\phi_{3}[\phi_{2}X,X']-\phi_{3}[X,\phi_{2}X'] \label{intermedio5}\\
&=[X,X']-[\phi_{2}X,\phi_{2}X']-\phi_{3}[\phi_{2}X,X']-\phi_{3}[X,\phi_{2}X'].\nonumber
\end{align}
Using \eqref{intermedio5} we have that
\begin{align*}
N^{(1)}_{\phi_3}(X,X')&=-[X,X']+[\phi_{3}\phi_{1}X,\phi_{3}\phi_{1}X']-\phi_{3}[\phi_{3}\phi_{1}X,X']-\phi_{3}[X,\phi_{3}\phi_{1}X']\\
&=-[X,X']+[\phi_{2}X,\phi_{2}X']+\phi_{3}[\phi_{2}X,X']+\phi_{3}[X,\phi_{2}X']=0.
\end{align*}
Arguing in the same way one can prove that
$N^{(1)}_{\phi_{3}}(Y,Y')=0$ for all $Y,Y'\in\Gamma({\mathcal
D}_{1}^{-})$. Next, for any $X\in\Gamma({\mathcal D}_{1}^{+})$ and
$X\in\Gamma({\mathcal D}_{1}^{-})$, by \eqref{lemma2} we get
\begin{equation*}
\phi_{3}N^{(1)}_{\phi_3}(X,Y)=-N^{(1)}_{\phi_3}(\phi_{3}X,Y)+2\left(d\eta(\phi_{3}X,Y)+d\eta(X,\phi_{3}Y)\right)\xi=0,
\end{equation*}
because $\phi_{3}{\mathcal D}_{1}^{\pm}={\mathcal D}^{\mp}$ and by
\eqref{legendrian}. On the other hand, since the almost
bi-paracontact structure $(\phi_{1},\phi_{2},\phi_{3})$ is
integrable, in particular Legendrian,
$\eta(N^{(1)}_{\phi_3}(X,Y))=-\eta([X,Y])+\eta([\phi_{3}X,\phi_{3}Y])=N^{(2)}_{\phi_{3}}(X,\phi_{3}Y)=0$
by Proposition \ref{proprieta4}. Therefore, as ${\mathcal
D}={\mathcal D}_{1}^{+}\oplus{\mathcal D}_{1}^{-}$, we conclude that
$N^{(1)}_{\phi_3}(Z,Z')=0$ for any $Z,Z'\in\Gamma({\mathcal D})$.
\end{proof}

A notion stronger than integrability is that of ``normal almost
bi-paracontact structure''.

\begin{definition}
Let $(M,\eta,\phi_1,\phi_2,\phi_3)$ be an almost bi-paracontact
manifold. If the almost paracontact structures
$(\phi_{1},\xi,\eta)$, $(\phi_{2},\xi,\eta)$ and the almost contact
structure $(\phi_{3},\xi,\eta)$ are normal, i.e.
$N^{(1)}_{\phi_\alpha}=0$ for each $\alpha\in\left\{1,2,3\right\}$,
$(\phi_1,\phi_2,\phi_3)$ is called a \emph{normal almost
bi-paracontact structure}.
\end{definition}

By arguing as in Corollary \ref{proprieta7} one can prove that if
$N^{(1)}_{\phi_1}$ and $N^{(1)}_{\phi_2}$ vanish identically, then
also $N^{(1)}_{\phi_3}=0$ and the almost bi-paracontact structure is
normal. Moreover, since, for each $\alpha\in\left\{1,2\right\}$ and
any $X\in\Gamma({\mathcal D})$,
$N^{(1)}_{\phi_\alpha}(\xi,X)=N^{(3)}_{\phi_\alpha}(\phi_{\alpha}X)$,
using Corollary \ref{proprieta7} one can prove the following
proposition.

\begin{proposition}\label{normality}
An almost bi-paracontact structure is normal if and only if it is
integrable and $N^{(3)}_{\phi_{1}}$ and $N^{(3)}_{\phi_{2}}=0$
vanish identically.
\end{proposition}

As a consequence we are able to give a geometrical interpretation to
 normality  in terms of Legendre foliations.

\begin{corollary}
An almost bi-paracontact structure is normal if and only, for each
$\alpha\in\left\{1,2\right\}$, both ${\mathcal D}_{\alpha}^{+}$ and
${\mathcal D}_{\alpha}^{-}$ are flat Legendre foliations.
\end{corollary}
\begin{proof}
Taking the definition of $N^{(3)}_{\phi_\alpha}$ into account, one
can easily prove that $\xi$ is foliate with respect both to
${\mathcal D}_{\alpha}^{+}$ and ${\mathcal D}_{\alpha}^{-}$ if and
only if $N^{(3)}_{\phi_\alpha}=0$. Then the assertion follows from
this remark and  Proposition \ref{normality}.
\end{proof}

Thus we have seen that, under some natural assumptions, an almost
bi-paracontact structure on a contact manifold gives rise to a pair
of transverse (almost) bi-Legendrian structures $({\mathcal
D}_{1}^{+},{\mathcal D}_{1}^{-})$ and  $({\mathcal
D}_{2}^{+},{\mathcal D}_{2}^{-})$. Conversely we have the following
result.

\begin{proposition}\label{legendre2}
Let $(L,Q)$ and $(L',Q')$ be two transverse almost bi-Legendrian
structures on the contact manifold $(M,\eta)$. Then there exists a
Legendrian almost bi-paracontact structure
$(\phi_{1},\phi_{2},\phi_{3})$ such that $L$, $Q$ and $L'$, $Q'$
are, respectively, the eigendistributions of $\phi_1$ and $\phi_2$.
\end{proposition}
\begin{proof}
We define $\phi_{1}|_{L}=I$, $\phi_{1}|_{Q}=-I$, $\phi_{1}\xi=0$ and
$\phi_{2}|_{L'}=I$, $\phi_{2}|_{Q'}=-I$, $\phi_{2}\xi=0$. Then we
set $\phi_{3}:=\phi_{1}\phi_{2}$. One can easily check that
$(\phi_1,\phi_2,\phi_3)$ is in fact an almost bi-paracontact
structure on $(M,\eta)$ such that, by construction, ${\mathcal
D}_{1}^{+}=L$, ${\mathcal D}_{1}^{-}=Q$ and ${\mathcal
D}_{2}^{+}=L'$, ${\mathcal D}_{2}^{-}=Q'$. In particular,
$(\phi_1,\phi_2,\phi_3)$ is Legendrian and it is integrable if and
only if $L$, $Q$, $L'$, $Q'$ are involutive.
\end{proof}

\section{Canonical connections on bi-paracontact manifolds}

In this section we attach to any almost bi-paracontact manifold some
canonical connections and then we study their nice properties. To
this end, we prove the following preliminary lemma.

\begin{lemma}\label{lemma1}
Let $(\phi_1,\phi_2,\phi_3)$ be an almost bi-paracontact structure
on the contact manifold $(M,\eta)$. For each
$\alpha\in\left\{1,2,3\right\}$ let $h_{\alpha}$ be the operator
defined by $h_{\alpha}:=\frac{1}{2}{\mathcal
L}_{\xi}\phi_{\alpha}=\frac{1}{2}N^{(3)}_{\phi_\alpha}$. Then
\begin{enumerate}
    \item[(a)] $h_{\alpha}\phi_{\alpha}=-\phi_{\alpha}h_{\alpha}$ for each $\alpha\in\left\{1,2,3\right\}$,
    \item[(b)] $\phi_{1}h_{2}+h_{1}\phi_{2}=h_{3}=-h_{2}\phi_{1}-\phi_{2}h_{1}$,\\
                $\phi_{1}h_{3}+h_{1}\phi_{3}=h_{2}=-h_{3}\phi_{1}-\phi_{3}h_{1}$, \\
                $\phi_{2}h_{3}+h_{2}\phi_{3}=-h_{1}=-h_{3}\phi_{2}-\phi_{3}h_{2}$.
\end{enumerate}
\end{lemma}
\begin{proof}
(a) \ Let us assume that $\alpha\in\left\{1,2\right\}$. Then $({\mathcal L}_{\xi}\phi_{\alpha})\circ\phi_{\alpha}+\phi_{\alpha}\circ({\mathcal L}_{\xi}\phi_{\alpha})={\mathcal L}_{\xi}(\phi_{\alpha}^2)={\mathcal L}_{\xi}(I-\eta\otimes\xi)=-({\mathcal L}_{\xi}\eta)\otimes\xi-\eta\otimes({\mathcal L}_{\xi}\xi)=0$, since ${\mathcal L}_{\xi}\eta=i_{\xi}d\eta+d i_{\xi}\eta=0$ by \eqref{contatto4}. Thus $h_{\alpha}\circ\phi_{\alpha}=-\phi_{\alpha}\circ h_{\alpha}$. The case $\alpha=3$ is similar.\\
(b) \ $2h_{3}={\mathcal L}_{\xi}\phi_{3}={\mathcal
L}_{\xi}(\phi_{1}\phi_{2})=({\mathcal
L}_{\xi}\phi_{1})\phi_{2}+\phi_{1}({\mathcal
L}_{\xi}\phi_{2})=2h_{1}\phi_{2}+2\phi_{1}h_{2}$. The other
equalities can be proved in a similar way.
\end{proof}

\begin{theorem}\label{connessioni}
Let $(\phi_1,\phi_2,\phi_3)$ be an almost bi-paracontact structure
on the contact manifold $(M,\eta)$. For each
$\alpha\in\left\{1,2,3\right\}$ there exists a unique linear
connection $\nabla^\alpha$ on $M$ satisfying the following
properties:
\begin{enumerate}
  \item[(i)] $\nabla^\alpha \xi=0$,
  \item[(ii)] $\nabla^{1}\phi_{1}=0, \ \ \nabla^{1}\phi_{2}=\eta\otimes(2h_{2}-h_{1}\phi_{3}+\phi_{3}h_1), \ \ \nabla^{1}\phi_{3}=\eta\otimes(2h_{3}-h_{1}\phi_{2}+\phi_{2}h_{1})$,\\
  $\nabla^{2}\phi_{1}=\eta\otimes(2h_{1}+h_{2}\phi_{3}-\phi_{3}h_{2}), \ \ \nabla^{2}\phi_{2}=0, \ \ \nabla^{2}\phi_{3}=\eta\otimes(2h_{3}+h_{2}\phi_{1}-\phi_{1}h_{2})$,\\
  $\nabla^{3}\phi_{1}=\eta\otimes(2h_{1}-h_{3}\phi_{2}+\phi_{2}h_{3}), \ \ \nabla^{3}\phi_{2}=\eta\otimes(2h_{2}+h_{3}\phi_{1}-\phi_{1}h_3), \ \ \nabla^{3}\phi_{3}=0$,
  \item[(iii)]   $T^\alpha(\phi_\alpha
  X,Y)-T^\alpha(X,\phi_\alpha Y)=2\left(d\eta(\phi_\alpha X,Y)-d\eta(X,\phi_\alpha
  Y)\right)\xi+\eta(Y)h_\alpha X+\eta(X)h_\alpha Y$ for any $X,Y\in\Gamma(TM)$,
\end{enumerate}
where $T^\alpha$ denotes the torsion tensor field of
$\nabla^\alpha$. $\nabla^{1}$, $\nabla^{2}$, $\nabla^{3}$ are
explicitly given by
\begin{align}
\nonumber    \nabla^{1}_{X}Y&=\frac{1}{4}\bigl([X,Y]-[\phi_1
X,\phi_1 Y]+\phi_1[X,\phi_1 Y]-\phi_1[\phi_1 X,Y]+\phi_2[X,\phi_2
Y]-\phi_3[X,\phi_3 Y]\\
\label{espressione1}   &\quad+\phi_3[\phi_1 X,\phi_2
Y]-\phi_2[\phi_1 X,\phi_3
Y]+2\eta(X)\left(-h_{1}\phi_{1}Y+h_{2}\phi_{2}Y-h_{3}\phi_{3}Y\right)\\
\nonumber
&\quad+2\eta(Y)h_{1}\phi_{1}X-\eta([X,Y])\xi+\eta([\phi_1 X,\phi_1
Y])\xi\bigr)+X(\eta(Y))\xi,\\
\nonumber    \nabla^{2}_{X}Y&=\frac{1}{4}\bigl([X,Y]-[\phi_2
X,\phi_2 Y]+\phi_2[X,\phi_2 Y]-\phi_2[\phi_2 X,Y]+\phi_1[X,\phi_1
Y]-\phi_3[X,\phi_3 Y]\\
\label{espressione2}    &\quad-\phi_3[\phi_2 X,\phi_1
Y]+\phi_1[\phi_2 X,\phi_3
Y]+2\eta(X)\left(h_{1}\phi_{1}Y-h_{2}\phi_{2}Y-h_{3}\phi_{3}Y\right)\\
\nonumber
&\quad+2\eta(Y)h_{2}\phi_{2}X-\eta([X,Y])\xi+\eta([\phi_2 X,\phi_2
Y])\xi\bigr)+X(\eta(Y))\xi,\\
\nonumber
\nabla^{3}_{X}Y&=\frac{1}{4}\bigl([X,Y]+[\phi_{3}X,\phi_{3}Y]+\phi_1[X,\phi_1
Y]+\phi_2[X,\phi_2 Y]-\phi_3[X,\phi_3 Y]+\phi_3[\phi_3 X,Y]\\
\label{espressione3}   &\quad+\phi_2[\phi_3 X,\phi_1
Y]-\phi_1[\phi_3 X,\phi_2
Y]+2\eta(X)\left(h_{1}\phi_{1}Y+h_{2}\phi_{2}Y+h_{3}\phi_{3}Y\right)\\
\nonumber &\quad-2\eta(Y)h_{3}\phi_{3}X-\eta([X,Y])\xi-\eta([\phi_3
X,\phi_3 Y])\xi\bigr)+X(\eta(Y))\xi.
\end{align}
\end{theorem}
\begin{proof}
First of all we prove the uniqueness. Fix an
$\alpha\in\left\{1,2,3\right\}$ and suppose that $\nabla$ and
$\nabla'$ are two linear connections satisfying (i), (ii) and (iii).
Let us define the tensor $A:=\nabla-\nabla'$. For any
$X,Y\in\Gamma({\mathcal D})$, since both $\nabla$ and $\nabla'$
preserve the almost bi-paracontact structure, one has
\begin{equation}\label{unicita0}
A(X,\phi_\beta Y)=\phi_{\beta}A(X,Y)
\end{equation}
for each $\beta\in\left\{1,2,3\right\}$. Because of (i), we have
$A(X,\xi)=0$ for all $X\in\Gamma(TM)$. Next, for all
$Y\in\Gamma({\mathcal D})$,
\begin{align*}
A(\xi,Y)&=\nabla_{\xi}Y-\nabla'_{\xi}Y\\
&=\nabla_{Y}\xi+T(\xi,Y)+[\xi,Y]-\nabla'_{Y}\xi-T'(\xi,Y)-[\xi,Y]\\
&=\epsilon\left(T(\xi,\phi_{\alpha}^{2}Y)-T'(\xi,\phi_{\alpha}^{2}Y)\right)\\
&=\epsilon\bigl(T(\phi_{\alpha}\xi,\phi_{\alpha}Y)-2\left(d\eta(\phi_{\alpha}\xi,\phi_{\alpha}Y)-d\eta(\xi,\phi_{\alpha}^{2}Y)\right)\xi-\eta(\phi_{\alpha}^{2}Y)h_{\alpha}\xi-\eta(\xi)h_{\alpha}\phi_{\alpha}^{2}Y\\
&\quad-T'(\phi_{\alpha}\xi,\phi_{\alpha}Y)+2\left(d\eta(\phi_{\alpha}\xi,\phi_{\alpha}Y)-d\eta(\xi,\phi_{\alpha}^{2}Y)\right)\xi+\eta(\phi_{\alpha}^{2}Y)h_{\alpha}\xi+\eta(\xi)h_{\alpha}\phi_{\alpha}^{2}Y\bigr)=0,
\end{align*}
where we have applied (ii) and (iii), and we have put $\epsilon=1$
if $\alpha\in\{1,2\}$, $\epsilon=-1$ if $\alpha=3$. Further, from
(iii) it follows that $T(\phi_\alpha X,Y)-T(X,\phi_\alpha
Y)=T'(\phi_\alpha X,Y)-T'(X,\phi_\alpha Y)$, that is
$\nabla_{\phi_{\alpha}X}Y-\nabla_{Y}\phi_{\alpha}X-\nabla_{X}\phi_{\alpha}Y+\nabla_{\phi_{\alpha}Y}X=\nabla'_{\phi_{\alpha}X}Y-\nabla'_{Y}\phi_{\alpha}X-\nabla'_{X}\phi_{\alpha}Y+\nabla'_{\phi_{\alpha}Y}X$.
Consequently,
\begin{equation}\label{unicita1}
A(\phi_{\alpha}X,Y)-A(Y,\phi_{\alpha}X)-A(X,\phi_{\alpha}Y)+A(\phi_{\alpha}Y,X)=0.
\end{equation}
If in \eqref{unicita1} we take $X\in\Gamma({\mathcal
D}^{+}_{\alpha})$ and $Y\in\Gamma({\mathcal D}^{-}_{\alpha})$ we
obtain
\begin{equation}\label{unicita2}
A(X,Y)=A(Y,X).
\end{equation}
By virtue of (ii), for each $Z\in\Gamma({\mathcal D})$, $\nabla_{Z}$
and $\nabla'_{Z}$ preserve the distributions ${\mathcal
D}^{\pm}_{\alpha}$. Thus $A(X,Y)\in\Gamma({\mathcal
D}_{\alpha}^{-})$ and $A(Y,X)\in\Gamma({\mathcal D}_{\alpha}^{+})$.
This together with \eqref{unicita2} and 5. of Proposition
\ref{proprieta1} imply that
\begin{equation}\label{unicita3}
A(X,Y)=A(Y,X)=0.
\end{equation}
Now let us consider $X,X'\in\Gamma({\mathcal D}^{+}_{\alpha})$ and
let $\beta\in\left\{1,2\right\}$, $\beta\neq\alpha$.  Note that, by
1.--2. of Proposition \ref{proprieta1},
$\phi_{\beta}X'\in\Gamma({\mathcal D}^{-}_{\alpha})$. Then, by
\eqref{unicita0} and \eqref{unicita3}, $A(X,X')=A(X,\phi_{\beta}^2
X')=\phi_{\beta}A(X,\phi_{\beta}X')=0$. In a similar way one can
prove that $A(X',X)=0$. Thus the tensor $A$ vanishes identically and
so $\nabla$ and $\nabla'$ coincide.

In order to prove the existence, for each
$\alpha\in\left\{1,2,3\right\}$, of a connection $\nabla^{\alpha}$
satisfying (i), (ii), (iii), we distinguish the cases
$\alpha\in\left\{1,2\right\}$ and $\alpha=3$. Let us consider
$\alpha\in\left\{1,2\right\}$. First of all, we put, by definition,
$\nabla^{\alpha}\xi:=0$. Next, notice that by (iii) we have that
$T^{\alpha}(\phi_{\alpha}X,\xi)=h_{\alpha}X$, for all
$X\in\Gamma(TM)$. In particular, for any $X\in\Gamma({\mathcal D})$,
$T^{\alpha}(X,\xi)=T(\phi_{\alpha}^2
X,\xi)=h_{\alpha}\phi_{\alpha}X$. It follows that necessarily
\begin{equation}\label{proprieta8}
\nabla^{\alpha}_{\xi}X=-h_{\alpha}\phi_{\alpha}X+[\xi,X]
\end{equation}
for all $X\in\Gamma({\mathcal D})$. In particular,
\begin{equation}\label{esistenza1}
\nabla^{\alpha}_{\xi}X=\left\{
                         \begin{array}{ll}
                           \left[\xi,X\right]_{{\mathcal D}_{\alpha}^{+}}, & \hbox{if $X\in\Gamma({\mathcal D}_{\alpha}^{+})$;}\\
                           \left[\xi,X\right]_{{\mathcal D}_{\alpha}^{-}}, & \hbox{if $X\in\Gamma({\mathcal D}_{\alpha}^{-})$.}
                         \end{array}
                       \right.
\end{equation}
Further, for any $X\in\Gamma({\mathcal D}^{+}_{\alpha})$ and
$Y\in\Gamma({\mathcal D}^{-}_{\alpha})$,
\begin{align*}
T^{\alpha}(X,Y)&=T^{\alpha}(\phi_{\alpha}X,Y)\\
&=T^{\alpha}(X,\phi_{\alpha}Y)+2\left(d\eta(\phi_\alpha
X,Y)-d\eta(X,\phi_{\alpha}Y)\right)\xi\\
&=-T^{\alpha}(X,Y)+4d\eta(X,Y)\xi,
\end{align*}
from which it follows that $T^{\alpha}(X,Y)=2d\eta(X,Y)\xi$. Hence,
$2d\eta(X,Y)\xi=\nabla^{\alpha}_{X}Y-\nabla^{\alpha}_{Y}X-[X,Y]_{{\mathcal
D}_{\alpha}^{+}}-[X,Y]_{{\mathcal D}_{\alpha}^{-}}-\eta([X,Y])\xi$,
that is
\begin{equation}\label{esistenza2}
\nabla^{\alpha}_{X}Y-[X,Y]_{{\mathcal
D}_{\alpha}^{-}}=\nabla^{\alpha}_{Y}X-[X,Y]_{{\mathcal
D}_{\alpha}^{+}}.
\end{equation}
Since, due to (ii), $\nabla^{\alpha}_{X}Y\in\Gamma({\mathcal
D}_{\alpha}^{-})$ and $\nabla^{\alpha}_{Y}X\in\Gamma({\mathcal
D}_{\alpha}^{+})$, both the sides of \eqref{esistenza2} must vanish
and we conclude that
\begin{equation}\label{definizione1}
\nabla^{\alpha}_{X}Y=[X,Y]_{{\mathcal D}_{\alpha}^{-}}, \ \
\nabla^{\alpha}_{Y}X=[Y,X]_{{\mathcal D}_{\alpha}^{+}}.
\end{equation}
Moreover, taking 1.--2. of Proposition \ref{proprieta1} into
account, for any $X,X'\in\Gamma({\mathcal D}_{\alpha}^{+})$ we have
\begin{equation}\label{definizione2}
\nabla^{\alpha}_{X}X'=\nabla^{\alpha}_{X}\phi_{\beta}^{2}X'=\phi_{\beta}\nabla^{\alpha}_{X}\phi_{\beta}X'=\phi_{\beta}\left[X,\phi_{\beta}X'\right]_{{\mathcal
D}^{-}_{\alpha}}
\end{equation}
and, for any $Y,Y'\in\Gamma({\mathcal D}_{\alpha}^{-})$,
\begin{equation}\label{definizione3}
\nabla^{\alpha}_{Y}Y'=\nabla^{\alpha}_{Y}\phi_{\beta}^{2}Y'=\phi_{\beta}\nabla^{\alpha}_{Y}\phi_{\beta}Y'=\phi_{\beta}\left[Y,\phi_{\beta}Y'\right]_{{\mathcal
D}^{+}_{\alpha}},
\end{equation}
where $\beta\in\left\{1,2\right\}$, $\beta\neq\alpha$. Now we
decompose any $X,Y\in\Gamma(TM)$ as $X=X_{+}+X_{-}+\eta(X)\xi$ and
$Y=Y_{+}+Y_{-}+\eta(Y)\xi$, where $X_{+}, Y_{+}$ and $X_{-}, Y_{-}$
denote the projections onto the subbundles ${\mathcal
D}_{\alpha}^{+}$ and ${\mathcal D}_{\alpha}^{-}$ of $TM$,
respectively. Then by \eqref{esistenza2}, \eqref{definizione1},
\eqref{definizione2} and \eqref{definizione3}  we get
\begin{align}\label{esistenza3}
\nabla^{\alpha}_{X}Y&=\phi_{\beta}\left[X_{+},\phi_{\alpha}Y_{+}\right]_{{\mathcal
D}_{\alpha}^{-}}+\left[X_{+},Y_{-}\right]_{{\mathcal
D}_{\alpha}^{-}}+\left[X_{-},Y_{+}\right]_{{\mathcal
D}_{\alpha}^{+}}+\phi_{\beta}\left[X_{-},\phi_{\alpha}Y_{-}\right]_{{\mathcal
D}_{\alpha}^{+}}\nonumber\\
&\quad+X(\eta(Y))\xi+\eta(X)\left[\xi,Y_{+}\right]_{{\mathcal
D}_{\alpha}^{+}}+\eta(X)\left[\xi,Y_{-}\right]_{{\mathcal
D}_{\alpha}^{-}}.
\end{align}
Notice that, as one can easily check,
\begin{equation}\label{formuleproiezioni}
X_{+}=\frac{1}{2}\left(X+\phi_{\alpha}X-\eta(X)\xi\right), \ \ \
X_{-}=\frac{1}{2}\left(X-\phi_{\alpha}X-\eta(X)\xi\right).
\end{equation}
Then, applying \eqref{formuleproiezioni} to \eqref{esistenza3},
after some very long but straightforward computations, we get
\begin{align}\label{definizione4}
\nonumber    \nabla^{\alpha}_{X}Y&=X(\eta(Y))\xi+\frac{1}{4}\bigl(\left[X,Y\right]-\left[\phi_{\alpha}X,\phi_{\alpha}Y\right]-\phi_{\alpha}\left[\phi_{\alpha}X,Y\right]+\phi_{\alpha}\left[X,\phi_{\alpha}Y\right]+\phi_{\beta}\left[X,\phi_{\beta}Y\right]\\
\nonumber    &\quad-\phi_{\beta}\phi_{\alpha}\left[X,\phi_{\beta}\phi_{\alpha}Y\right]-\phi_{\beta}\phi_{\alpha}\left[\phi_{\alpha}X,\phi_{\beta}Y\right]+\phi_{\beta}\left[\phi_{\alpha}X,\phi_{\beta}\phi_{\alpha}Y\right]+\eta(X)\phi_{\alpha}\left[\xi,\phi_{\alpha}Y\right]\\
&\quad-\eta(Y)\phi_{\alpha}\left[\xi,\phi_{\alpha}X\right]-\eta(X)\phi_{\beta}\left[\xi,\phi_{\beta}Y\right]+\eta(X)\phi_{\beta}\phi_{\alpha}\left[\xi,\phi_{\beta}\phi_{\alpha}Y\right]+\eta(Y)[\xi,X]\\
\nonumber&\quad+\eta(X)[\xi,Y]-\eta([X,Y])\xi+\eta\left(\left[\phi_{\alpha}X,\phi_{\alpha}Y\right]\right)\xi-\eta(Y)\eta([\xi,X])\xi-\eta(X)\eta([\xi,Y])\xi\bigr).
\end{align}
Then we can take \eqref{definizione4} as a definition and  one can
easily check that, for each $\alpha\in\left\{1,2\right\}$,
$\nabla^{\alpha}$ satisfies (i), (ii) and (iii). Moreover, taking
the definition of the operators $h_1$, $h_2$, $h_3$ into account, it
is not difficult to verify that \eqref{definizione4} implies
\eqref{espressione1}--\eqref{espressione2}. \ It remains to prove
the theorem for $\alpha=3$. In that case the same construction as
for $\alpha\in\left\{1,2\right\}$ can be repeated, but now arguing
on the eigendistributions ${\mathcal D}^{+}_{3}$ and ${\mathcal
D}^{-}_{3}$ of $\phi_3$ corresponding to $i$ and $-i$, respectively,
and replacing \eqref{formuleproiezioni} with
\begin{equation*}
p_{{\mathcal
D}_{3}^{+}}=\frac{1}{2}\left(I-i\phi_{3}-\eta\otimes\xi\right), \ \
\ p_{{\mathcal
D}_{3}^{-}}=\frac{1}{2}\left(I+i\phi_{3}-\eta\otimes\xi\right).
\end{equation*}
Then after very long computations one obtains
\begin{align*}
\nabla^{3}_{X}Y&=X(\eta(Y))\xi+\frac{1}{4}\bigl(\left[X,Y\right]+\left[\phi_{3}X,\phi_{3}Y\right]+\phi_{1}\left[X,\phi_{1}Y\right]+\phi_{2}\left[X,\phi_{2}Y\right]-\phi_{3}\left[X,\phi_{3}Y\right]\\
&\quad+\phi_{3}[\phi_{3}X,Y]+\phi_{2}\left[\phi_{3}X,\phi_{1}Y\right]-\phi_{1}\left[\phi_{3}X,\phi_{2}Y\right]-\eta(X)\phi_{1}\left[\xi,\phi_{1}Y\right]+\eta(Y)\phi_{3}\left[\xi,\phi_{3}X\right]\\
&\quad-\eta(X)\phi_{2}\left[\xi,\phi_{2}Y\right]-\eta(X)\phi_{3}\left[\xi,\phi_{3}Y\right]-\eta(Y)[\xi,\phi_{3}^{2}X]+\eta(X)[\xi,\phi_{1}^{2}Y]-\eta([X,Y])\xi\\
&\quad-\eta\left(\left[\phi_{3}X,\phi_{3}Y\right]\right)\xi\bigr),
\end{align*}
from which \eqref{espressione3} follows.
\end{proof}

\begin{proposition}\label{torsione7}
The torsion tensor fields of the linear connections $\nabla^{1}$,
$\nabla^{2}$, $\nabla^{3}$ stated in Theorem \ref{connessioni} are
given by
\begin{align}
\nonumber    T^{1}(X,Y)&=\frac{1}{4}\bigl(\bigl(N^{(1)}_{\phi_3}-N^{(1)}_{\phi_2}\bigr)(X,Y)+\bigl(N^{(1)}_{\phi_3}-N^{(1)}_{\phi_2}\bigr)(\phi_{1}X,\phi_{1}Y)\bigr)\\
\label{torsione1}&\quad+\bigl(d\eta(X,Y)-d\eta(\phi_{1}X,\phi_{1}Y)\bigr)\xi\\
\nonumber    &\quad+\frac{1}{2}\bigl(\eta(X)\bigl(-2h_{1}\phi_{1}Y+h_{2}\phi_{2}Y-h_{3}\phi_{3}Y\bigr)-\eta(Y)\bigl(-2h_{1}\phi_{1}X+h_{2}\phi_{2}X-h_{3}\phi_{3}X\bigr)\bigr),\\
\nonumber    T^{2}(X,Y)&=\frac{1}{4}\bigl(\bigl(N^{(1)}_{\phi_3}-N^{(1)}_{\phi_1}\bigr)(X,Y)+\bigl(N^{(1)}_{\phi_3}-N^{(1)}_{\phi_1}\bigr)(\phi_{2}X,\phi_{2}Y)\bigr)\\
\label{torsione2}&\quad+\bigl(d\eta(X,Y)-d\eta(\phi_{2}X,\phi_{2}Y)\bigr)\xi\\
\nonumber   &\quad+\frac{1}{2}\bigl(\eta(X)\bigl(h_{1}\phi_{1}Y-2h_{2}\phi_{2}Y-h_{3}\phi_{3}Y\bigr)-\eta(Y)\bigl(h_{1}\phi_{1}X-2h_{2}\phi_{2}X-h_{3}\phi_{3}X\bigr)\bigr),\\
\nonumber   T^{3}(X,Y)&=-\frac{1}{4}\bigl(\bigl(N^{(1)}_{\phi_1}+N^{(1)}_{\phi_1}\bigr)(X,Y)-\bigl(N^{(1)}_{\phi_1}+N^{(1)}_{\phi_2}\bigr)(\phi_{3}X,\phi_{3}Y)\bigr)\\
\label{torsione3}&\quad+\bigl(d\eta(X,Y)+d\eta(\phi_{3}X,\phi_{3}Y)\bigr)\xi\\
\nonumber
&\quad+\frac{1}{2}\bigl(\eta(X)\bigl(h_{1}\phi_{1}Y+h_{2}\phi_{2}Y+2h_{3}\phi_{3}Y\bigr)-\eta(Y)\bigl(h_{1}\phi_{1}X+h_{2}\phi_{2}X+2h_{3}\phi_{3}X\bigr)\bigr).
\end{align}
\end{proposition}
\begin{proof}
The proof follows  from \eqref{espressione1}--\eqref{espressione3}
by a straightforward computation.
\end{proof}

The connections stated in Theorem \ref{connessioni} give rise to a
canonical connection on an almost bi-paracontact manifold that can
be considered as an odd-dimensional counterpart  of the Obata
connection of an anti-hypercomplex (or complex-product) manifold
(cf. \cite{andrada}, \cite{marchiafava}, \cite{santamaria},
\cite{yano2}).

\begin{theorem}\label{connessioneobata}
Let $(M,\eta,\phi_{1},\phi_{2},\phi_{3})$ be an almost
bi-paracontact manifold. There exists a unique linear connection
$\nabla^{c}$ on $M$ such that
\begin{enumerate}
  \item[(i)] $\nabla^{c}\xi=0$,
  \item[(ii)] $\nabla^{c}\phi_{\alpha}=\frac{2}{3}\eta\otimes h_{\alpha}$ for each $\alpha\in\left\{1,2,3\right\}$,
  \item[(iii)]
$T^{c}=d\eta+\frac{1}{3}\bigl(-d\eta(\phi_1\cdot,\phi_1\cdot)-d\eta(\phi_2\cdot,\phi_2\cdot)+d\eta(\phi_3\cdot,\phi_3\cdot)\bigr)+\frac{1}{6}\bigl(-N^{(1)}_{\phi_1}-N^{(1)}_{\phi_2}+N^{(1)}_{\phi_3}\bigr)$.
\end{enumerate}
\end{theorem}
\begin{proof}
We first prove the uniqueness of a linear connection satisfying the
conditions (i), (ii) and (iii). Let $\nabla$ and $\nabla'$ be two
linear connections satisfying (i), (ii), (iii). Let us define the
tensor $A:=\nabla-\nabla'$. Because the expressions of the torsion
tensor fields of $\nabla$ and $\nabla'$ coincide, one has
immediately that $A(X,Y)=A(Y,X)$ for all $X,Y\in\Gamma(TM)$. Hence
$A$ is symmetric. Then, due to (ii), one has
$A(X,\phi_{1}Y)=\phi_{1}A(X,Y)=\phi_{1}A(Y,X)=A(Y,\phi_{1}X)=A(\phi_{1}X,Y)$
and, analogously, $A(X,\phi_{2}Y)=\phi_{2}A(X,Y)=A(\phi_{2}X,Y)$.
Therefore
\begin{equation}\label{comp1}
A(\phi_{1}X,\phi_{2}Y)=\phi_{1}A(X,\phi_{2}Y)=\phi_{1}\phi_{2}A(X,Y)=\phi_{3}A(X,Y).
\end{equation}
On the other hand
\begin{equation}\label{comp2}
A(\phi_{1}X,\phi_{2}Y)=\phi_{2}A(\phi_{1}X,Y)=\phi_{2}\phi_{1}A(X,Y)=-\phi_{3}A(X,Y).
\end{equation}
Thus comparing \eqref{comp1} and \eqref{comp2} we get
$\phi_{3}A(X,Y)=-\phi_{3}A(X,Y)$. Applying $\phi_3$ to both the
sides of the previous identity we obtain
\begin{equation}\label{comp3}
-A(X,Y)+\eta(A(X,Y))\xi=A(X,Y)-\eta(A(X,Y))\xi.
\end{equation}
Notice that as, for each $Z\in\Gamma({\mathcal D})$, $\nabla_{Z}$
and $\nabla'_{Z}$ preserve $\phi_1$, they also preserve the
eigendistributions ${\mathcal D}_{1}^{\pm}$ and hence the contact
distribution ${\mathcal D}={\mathcal D}_{1}^{+}\oplus{\mathcal
D}_{1}^{-}$. This implies that $\eta(A(X,Y))=0$ whenever
$X,Y\in\Gamma({\mathcal D})$. Moreover, $A(X,\xi)=0$ and
$A(\xi,Y)=A(\xi,\phi_{1}^{2}Y)=A(\phi_{1}\xi,\phi_{1}Y)=0$.
Consequently \eqref{comp3} yields that $A$ is anti-symmetric. Since
it is also symmetric, it necessarily vanishes identically. This
proves that $\nabla=\nabla'$.

In order to define a (necessarily unique) linear connection
satisfying the conditions (i), (ii), (iii), we consider the
barycenter of the canonical connections $\nabla^1$, $\nabla^2$,
$\nabla^3$ stated in Theorem \ref{connessioni}. Thus we define, for
all $X,Y\in\Gamma(TM)$,
\begin{equation*}
\nabla^{c}_{X}Y:=\frac{1}{3}\left(\nabla^{1}_{X}Y+\nabla^{2}_{X}Y+\nabla^{3}_{X}Y\right).
\end{equation*}
We have immediately that $\nabla^{c}\xi=0$. By the expressions in
(ii) of Theorem \ref{connessioni} and by (b) of Lemma \ref{lemma1}
we have
\begin{equation*}
\nabla^{c}\phi_{1}=\frac{1}{3}\left(\nabla^{2}\phi_1+\nabla^{3}\phi_1\right)=\frac{1}{3}\eta\otimes\left(2h_{1}+h_{2}\phi_{3}-\phi_{3}h_{2}+2h_{1}-h_{3}\phi_{2}+\phi_{2}h_{3}\right)=\frac{2}{3}\eta\otimes
h_{1}
\end{equation*}
and, analogously, $\nabla^{c}\phi_{2}=\frac{2}{3}\eta\otimes h_2$,
$\nabla^{c}\phi_{3}=\frac{2}{3}\eta\otimes h_3$. Using
\eqref{torsione1}--\eqref{torsione3} we can easily find the
expression of the torsion of $\nabla^{c}$:
\begin{align}
T^{c}(X,Y)&=T^{1}(X,Y)+T^{2}(X,Y)+T^{3}(X,Y)\nonumber\\
&=d\eta(X,Y)\xi+\frac{1}{3}\left(-d\eta(\phi_{1}X,\phi_{1}Y)-d\eta(\phi_{2}X,\phi_{2}Y)+d\eta(\phi_{3}X,\phi_{3}Y)\right)\xi\label{torsioneobata2}\\
&\quad+\frac{1}{6}\left(-N^{(1)}_{\phi_1}(X,Y)-N^{(1)}_{\phi_2}(X,Y)+N^{(1)}_{\phi_{3}}(X,Y)\right).\nonumber
\end{align}
\end{proof}

The unique connection $\nabla^{c}$ stated in Theorem
\ref{connessioneobata} will be called the \emph{canonical
connection} of the almost bi-paracontact manifold
$(M,\eta,\phi_{1},\phi_{2},\phi_{3})$. Using
\eqref{espressione1}--\eqref{espressione3}, after a long
computation, one finds that the explicit expression of $\nabla^{c}$
is the following:
\begin{align*}
\nabla^{c}_{X}Y&=\frac{1}{12}\bigl(3[X,Y]-[\phi_{1}X,\phi_{1}Y]-[\phi_{2}X,\phi_{2}Y]+[\phi_{3}X,\phi_{3}Y]+3\phi_{1}[X,\phi_{1}Y]+3\phi_{2}[X,\phi_{2}Y]\\
&\quad-3\phi_{3}[X,\phi_{3}Y]-\phi_{1}[\phi_{1}X,Y]-\phi_{2}[\phi_{2}X,Y]+\phi_{3}[\phi_{3}X,Y]+\phi_{1}[\phi_{2}X,\phi_{3}Y]\\
&\quad-\phi_{1}[\phi_{3}X,\phi_{2}Y]-\phi_{2}[\phi_{1}X,\phi_{3}Y]+\phi_{2}[\phi_{3}X,\phi_{1}Y]+\phi_{3}[\phi_{1}X,\phi_{2}Y]-\phi_{3}[\phi_{2}X,\phi_{1}Y]\\
&\quad+2\eta(X)(h_{1}\phi_{1}Y+h_{2}\phi_{2}Y-h_{3}\phi_{3}Y)+2\eta(Y)(h_{1}\phi_{1}X+h_{2}\phi_{2}X-h_{3}\phi_{3}X)\\
&\quad+\left(\eta([\phi_{1}X,\phi_{1}Y])+\eta([\phi_{2}X,\phi_{2}Y])-\eta([\phi_{3}X,\phi_{3}Y])-3\eta([X,Y])\right)\xi\bigr)+X(\eta(Y))\xi.
\end{align*}

\begin{corollary}
Let $(M,\eta,\phi_1,\phi_2,\phi_3)$ be a normal almost
bi-paracontact manifold.
\begin{enumerate}
  \item[1.] There exists a unique linear connection
$\nabla^{c}$ on $M$ preserving the almost bi-paracontact structure
and whose torsion is given by
\begin{equation}\label{torsionesemplificata}
T^{c}=2d\eta\otimes\xi.
\end{equation}
  \item[2.] The curvature tensor field of $\nabla^{c}$ satisfies
\begin{equation}\label{curvatura1}
R^{c}(\phi_{1}\cdot,\phi_{1}\cdot)=R^{c}(\phi_{2}\cdot,\phi_{2}\cdot)=-R^{c}(\phi_{3}\cdot,\phi_{3}\cdot)=-R^{c}.
\end{equation}
In particular, for all $X\in\Gamma(TM)$
\begin{equation}\label{curvatura4}
R^{c}(X,\xi)=0
\end{equation}
  \item[3.] The Ricci tensor of $\nabla^{c}$, defined as \
$\emph{Ric}^{c}(X,Y):=\emph{trace}(Z\mapsto R^{c}(Z,X)Y)$, is given
by
\begin{equation}\label{ricci1}
\emph{Ric}^{c}(X,Y)=-\frac{1}{2}\emph{trace}(R^{c}(X,Y)).
\end{equation}
In particular, the Ricci tensor is skew-symmetric and \
$\emph{Ric}^{c}(\phi_{1}\cdot,\phi_{1}\cdot)=\emph{Ric}^{c}(\phi_{2}\cdot,\phi_{2}\cdot)$
$=-\emph{Ric}^{c}(\phi_{3}\cdot,\phi_{3}\cdot)=\emph{Ric}^{c}$.
 \item[4.] The connection $\nabla^{c}$ and the connections
$\nabla^{1}$, $\nabla^{2}$, $\nabla^{3}$ coincide.
\end{enumerate}
\end{corollary}
\begin{proof}
1. As in any normal almost bi-paracontact manifold the tensor fields
$h_{1}$, $h_{2}$, $h_{3}$ vanish identically, by (ii) of Theorem
\ref{connessioneobata}, $\nabla^{c}$ preserves the tensor fields
$\phi_{1}$, $\phi_{2}$, $\phi_{3}$. Moreover, by \eqref{legendrian}
the expression \eqref{torsioneobata2} of the torsion simplifies in
\eqref{torsionesemplificata}.\\
2. First of all notice that, since $\nabla^{c}\phi_{\alpha}=0$, for
each $\alpha\in\left\{1,2,3\right\}$ we have
\begin{equation}\label{curvatura2}
R^{c}(X,Y)\circ\phi_{\alpha}=\phi_\alpha\circ R^{c}(X,Y).
\end{equation}
for all $X,Y\in\Gamma(TM)$. Now the Bianchi identity yields
\begin{align}\label{bianchi1}
R^{c}(X,Y)Z+R^{c}(Y,Z)X+R^{c}(Z,X)Y&=T^{c}(T^{c}(X,Y),Z)+(\nabla^{c}_{X}T^{c})(Y,Z)+T^{c}(T^{c}(Y,Z),X)\nonumber\\
&\quad+(\nabla^{c}_{Y}T^{c})(Z,X)+T^{c}(T^{c}(Z,X),Y)+(\nabla^{c}_{Z}T^{c})(X,Y).
\end{align}
We examine the terms in the right-hand-side of \eqref{bianchi1}.
Notice that, by \eqref{torsionesemplificata},
$T^{c}(T^{c}(X,Y),Z)=4{d\eta(X,Y)}{d\eta(\xi,Z)}\xi=0$ and
\begin{align*}
(\nabla^{c}_{X}T^{c})(Y,Z)&=\nabla^{c}_{X}(2d\eta(Y,Z)\xi)-2d\eta(\nabla^{c}_{X}Y,Z)\xi-2d\eta(Y,\nabla^{c}_{X}Z)\xi\\
&=2X(d\eta(Y,Z))\xi+2d\eta(Y,Z)\nabla^{c}_{X}\xi-2d\eta(\nabla^{c}_{X}Y,Z)\xi-2d\eta(Y,\nabla^{c}_{X}Z)\xi\\
&=2(\nabla^{c}_{X}d\eta)(Y,Z)\xi.
\end{align*}
Hence \eqref{bianchi1} simplifies in
\begin{align}\label{bianchi2}
R^{c}(X,Y)Z+R^{c}(Y,Z)X+R^{c}(Z,X)Y&=2\bigl((\nabla^{c}_{X}d\eta)(Y,Z)+(\nabla^{c}_{Y}d\eta)(Z,X)\nonumber\\
&\quad\quad+(\nabla^{c}_{Z}d\eta)(X,Y)\bigr)\xi.
\end{align}
Now in \eqref{bianchi2} consider $X,Y\in\Gamma({\mathcal
D}_{\alpha}^{+})$ and $Z\in\Gamma({\mathcal D}_{\alpha}^{-})$,
$\alpha\in\left\{1,2\right\}$. Then, as $\nabla^{c}$ preserves the
contact distribution, the left-hand-side of \eqref{bianchi2} is
tangent to $\mathcal D$ whereas the right-hand-side is transversal
to $\mathcal D$. Hence they both vanish. Thus, in particular
\begin{equation}\label{bianchi3}
R^{c}(X,Y)Z=-R^{c}(Y,Z)X-R^{c}(Z,X)Y.
\end{equation}
But the left-hand-side of \eqref{bianchi3} is a section of
${\mathcal D}_{\alpha}^{-}$, whereas the right-hand-side is a
section of ${\mathcal D}_{\alpha}^{+}$. Consequently,
$R^{c}(X,Y)Z=0$ for all $X,Y\in\Gamma({\mathcal D}_{\alpha}^{+})$
and $Z\in\Gamma({\mathcal D}_{\alpha}^{-})$. Since by Proposition
\ref{proprieta1}, for any $\beta\neq\alpha$, $\phi_\beta$ maps
${\mathcal D}_{\alpha}^{-}$ onto ${\mathcal D}_{\alpha}^{+}$,
applying \eqref{curvatura2} we get that $R^{c}(X,Y)Z=0$ also for
$Z\in\Gamma({\mathcal D}_{\alpha}^{+})$. Moreover, obviously,
$R^{c}(X,Y)\xi=0$, so that we can conclude that
\begin{equation}\label{curvatura3}
R^{c}(X,Y)=0
\end{equation}
for any $X,Y\in\Gamma({\mathcal D}_{\alpha}^{+})$. In a similar way
one can prove that \eqref{curvatura3} holds for
$X,Y\in\Gamma({\mathcal D}_{\alpha}^{-})$. Thus in both cases the
relation $R^{c}(\phi_{\alpha}X,\phi_{\alpha}Y)=-R^{c}(X,Y)$,
$\alpha\in\left\{1,2\right\}$, is trivially satisfied. Moreover, if
$X\in\Gamma({\mathcal D}_{\alpha}^{+})$ and $Y\in\Gamma({\mathcal
D}_{\alpha}^{-})$,
$R^{c}(\phi_{\alpha}X,\phi_{\alpha}Y)=R^{c}(X,-Y)=-R^{c}(X,Y)$. In
order to complete the proof in the case
$\alpha\in\left\{1,2\right\}$ it remains to prove that
$R^{c}(X,\xi)=0$ for any $X\in\Gamma({\mathcal D})$. Notice that, as
$\nabla^{c}\xi=0$ and $T^{c}(X,\xi)=2d\eta(X,\xi)=0$,
$\nabla^{c}_{\xi}X=[\xi,X]$. By applying again the Bianchi identity
\eqref{bianchi1} we obtain, for all $Z\in\Gamma({\mathcal D})$,
\begin{align*}
R^{c}(X,\xi)Z+R^{c}(\xi,Z)X&=(\nabla^{c}_{\xi}T^{c})(Z,X)\\
&=\nabla^{c}_{\xi}(T^{c}(Z,X))-T^{c}([\xi,Z],X)-T^{c}(Z,[\xi,X])\\
&=2({\mathcal L}_{\xi}d\eta)(Z,X) \xi=0.
\end{align*}
Consequently $R^{c}(X,\xi)Z=-R^{c}(\xi,Z)X$. If in the last equality
we take $X\in\Gamma({\mathcal D}_{\alpha}^{+})$ and
$Z\in\Gamma({\mathcal D}_{\alpha}^{-})$, the left-hand-side is a
section of ${\mathcal D}_{\alpha}^{-}$ while the right-hand-side is
a section of ${\mathcal D}_{\alpha}^{+}$. Thus they both vanish and
taking \eqref{curvatura2} into account we conclude that
$R^{c}(X,\xi)=0$ for all $X\in\Gamma({\mathcal D})$. Finally, for
any $X,Y\in\Gamma(TM)$,
$R^{c}(\phi_{3}X,\phi_{3}Y)=R^{c}(\phi_{1}\phi_{2}X,\phi_{1}\phi_{2}Y)=-R^{c}(\phi_{2}X,\phi_{2}Y)=R^{c}(X,Y)$.\\
3. For simplifying the notation, let $r_{XY}$ denote the
endomorphism $Z\mapsto R^{c}(Z,X)Y$, so that
$\textrm{Ric}^{c}(X,Y)=\textrm{trace}(r_{XY})$. From
\eqref{curvatura1} it follows immediately that $r_{XY}(\xi)=0$. Let
$\{E_{1},\ldots,E_{n},E_{n+1},\ldots,E_{2n},\xi\}$ be a local basis
such that, for each $i\in\left\{1,\ldots,n\right\}$,
$E_{i}\in\Gamma({\mathcal D}_{1}^{+})$ and
$E_{n+i}=\phi_{2}E_{i}\in\Gamma({\mathcal D}_{1}^{-})$.  In order to
prove \eqref{ricci1} we distinguish the cases (i)
$X,Y\in\Gamma({\mathcal D}_{1}^{+})$, (ii) $X,Y\in\Gamma({\mathcal
D}_{1}^{-})$, (iii) $X\in\Gamma({\mathcal D}_{1}^{+})$,
$Y\in\Gamma({\mathcal D}_{1}^{-})$, (iv) $X\in\Gamma(TM)$, $Y=\xi$.
In the first case, due to \eqref{curvatura3},
$r_{XY}(E_{i})=R^{c}(E_{i},X)Y=0$. Moreover,
$r_{XY}(E_{n+i})=R^{c}(E_{n+i},X)Y\in\Gamma({\mathcal D}_{1}^{+})$
so that it has no components along the direction of
$E_{n+1},\ldots,E_{2n},\xi$. Hence
$\textrm{Ric}^{c}(X,Y)=\textrm{trace}(r_{XY})=0$. On the other hand,
since $R^{c}(X,Y)=0$, also the right-hand-side of \eqref{ricci1}
vanishes. The case (ii) being analogous, we pass to the case (iii).
First of all, by \eqref{curvatura3}, $r_{XY}(E_i)=R^{c}(E_i,X)Y=0$.
Next, by the Bianchi identity used before,
$r_{XY}(E_{n+i})=R^{c}(E_{n+i},X)Y=-R^{c}(X,Y)E_{n+i}-R^{c}(Y,E_{n+i})X=-R^{c}(X,Y)E_{n+i}$,
as $R^{c}(Y,E_{n+i})=0$. Since
$R^{c}(X,Y)E_{n+i}=R^{c}(X,Y)\phi_{1}E_{i}=\phi_{1}(R^{c}(X,Y)E_{i})$,
we conclude that
$\textrm{trace}(r_{XY})=-\frac{1}{2}\textrm{trace}R^{c}(X,Y)$. The
last case is obvious since, due to \eqref{curvatura1},
$\textrm{Ric}^{c}(X,\xi)=0=-\frac{1}{2}\textrm{trace}(R^{c}(X,\xi))$.\\
4. Proposition \ref{torsione7}, \eqref{torsionesemplificata} and the
normality of the almost bi-paracontact structure imply that
$T^{1}(X,Y)=T^{2}(X,Y)=T^{3}(X,Y)=2d\eta(X,Y)\xi=T^{c}(X,Y)$ for all
$X,Y\in\Gamma(TM)$. Moreover, according to (ii) of Theorem
\ref{connessioni}, because of the vanishing of the tensor fields
$h_{1}$, $h_{2}$, $h_{3}$, each connection $\nabla^{1}$,
$\nabla^{2}$, $\nabla^{3}$ preserves the tensor fields $\phi_{1}$,
$\phi_{2}$, $\phi_{3}$. Consequently, for each
$\alpha\in\left\{1,2,3\right\}$, $\nabla^{\alpha}$ fulfils all the
conditions of Theorem \ref{connessioneobata} and hence coincides
with $\nabla^{c}$.
\end{proof}

\begin{corollary}
Every normal almost bi-paracontact manifold carries four mutually
transverse Legendre foliations whose leaves are totally geodesic and
admit an affine structure.
\end{corollary}
\begin{proof}
Since the almost bi-paracontact structure is normal, it is in
particular integrable, so that the eigendistributions $\mathcal
D_{1}^{+},\mathcal D_{1}^{-},\mathcal D_{2}^{+}, \mathcal D_{2}^{-}$
define four mutually transverse Legendre foliations on the manifold.
The leaves of these foliations are auto-parallel with respect to the
canonical connection $\nabla^{c}$, so that they are totally
geodesic. Finally, for each $\alpha\in\left\{1,2\right\}$, for any
$X,X'\in\Gamma({\mathcal D}_{\alpha}^{\pm})$ we have, by
\eqref{torsionesemplificata}, $T^{c}(X,X')=0$ and,  by
\eqref{curvatura1}, $R^{c}(X,X')=0$. Thus $\nabla^{c}$ induces a
flat, torsion-free connection on the leaves of the foliations
$\mathcal D_{1}^{+},\mathcal D_{1}^{-},\mathcal D_{2}^{+}, \mathcal
D_{2}^{-}$.
\end{proof}

We conclude the section by studying the transverse geometry of a
normal almost bi-paracontact manifold with respect to the Reeb
foliation. We show in fact that the space of leaves of a normal
almost bi-paracontact manifold is anti-hypercomplex (see
\cite{marchiafava} or, with different names, \cite{andrada},
\cite{santamaria}, \cite{yano2}). We recall that an
\emph{anti-hypercomplex structure} on an even dimensional manifold
is given by two anti-commuting product structures $I$, $J$ and a
complex structure $K$ such that $IJ=K$. Then one can prove that the
manifold admits a canonical connection, usually called the
\emph{Obata connection}, defined as the unique torsion-free
connection preserving the anti-hypercomplex structure.

\begin{theorem}
Let $(M,\phi_{1},\phi_{2},\phi_{3})$ be a normal almost
bi-paracontact manifold. Then the $1$-dimensional foliation defined
by the Reeb vector field $\xi$ is transversely anti-hypercomplex.
Furthermore, the canonical connection $\nabla^{c}$ is (locally)
projectable to the Obata connection defined on the leaf space.
\end{theorem}
\begin{proof}
First of all we have to prove that the tensor fields $\phi_{1}$,
$\phi_{2}$, $\phi_{3}$ are ``foliated'' objects, i.e. they are
constant along the leaves of the Reeb foliation ${\mathcal
F}_{\xi}$. Thus we have to show that ${\mathcal
L}_{\xi}\phi_{\alpha}=0$ for each $\alpha\in\left\{1,2,3\right\}$.
In fact this condition is satisfied because, by assumption,
$N^{(1)}_{\phi_{\alpha}}=0$, so that also
$N^{(3)}_{\phi_{\alpha}}={\mathcal L}_{\xi}\phi_{\alpha}=0$. Thus
the tensor fields $\phi_{1}$, $\phi_{2}$, $\phi_{3}$ are
projectable. We prove that they (locally) project onto an
anti-hypercomplex structure. Let $\pi$ be a local submersion
defining the Reeb foliation. For each
$\alpha\in\left\{1,2,3\right\}$ let $J_{\alpha}$ be the tensor field
defined by $\pi_{\ast}\circ\phi_{\alpha}=J_{\alpha}\circ\pi_{\ast}$.
Then it is clear that $(J_{1},J_{2},J_{3})$ is an almost
anti-hypercomplex structure. Moreover, for any two (local) vector
fields $X'$ and $Y'$ in the leaf space, denoting by $X$ and $Y$ the
unique basic vector fields on $M$ such such that $\pi_{\ast}X=X'$
and $\pi_{\ast}Y=Y'$, we have
\begin{equation*}
[J_{\alpha},J_{\alpha}](X',Y')=\pi_{\ast}\bigl(N^{(1)}_{\phi_{\alpha}}(X,Y)\bigr)=0,
\end{equation*}
so that the structure is integrable. For concluding the proof we
prove that the canonical connection $\nabla^{c}$ projects onto the
the Obata connection $\nabla^{Ob}$. First we prove that
 $\nabla^{c}$ is projectable, i.e. it projects to
connections of the local slice spaces of ${\mathcal F}_{\xi}$. The
conditions for this are: a) for any basic vector fields
$X\in\Gamma({\mathcal D})$ and for any $V\in\Gamma(T{\mathcal
F}_{\xi})$ one has $\nabla^{c}_{V}X=0$, b) if $X$ and $Y$ are basic
vector fields then also $\nabla^{c}_{X}Y$ is a basic vector field
(\cite{molino}). Here, by ``basic vector field'' we mean a vector
field $X$ transverse to the foliation ${\mathcal F}_{\xi}$ which is
locally projectable to a vector field on the leaf space by means a
local submersion defining  ${\mathcal F}_{\xi}$; one can see that
this is equivalent to require that $[X,V]$ is still tangent to the
foliation for any $V\in\Gamma(T{\mathcal F}_{\xi})$ (cf.
\cite{molino}, \cite{tondeur}). Now the condition (a) is easily
verified since $\nabla^{c}_{\xi}X=[\xi,X]=0$ because $[\xi,X]$ is
tangent both to $\mathcal D$ and to ${\mathcal F}_{\xi}$ ($X$ being
basic). Also the second condition holds. Indeed first recall that,
by construction, $\nabla^{c}$ preserves the contact distribution;
next, by \eqref{curvatura4},
\begin{equation}\label{ausiliare1}
0=R^{c}(X,\xi)Y=\nabla^{c}_{X}\nabla^{c}_{\xi}Y-\nabla^{c}_{\xi}\nabla^{c}_{X}Y-\nabla^{c}_{[X,\xi]}Y=\nabla^{c}_{X}[\xi,Y]-\nabla^{c}_{\xi}\nabla^{c}_{X}Y=-\nabla^{c}_{\xi}\nabla^{c}_{X}Y
\end{equation}
since $[X,\xi]=[Y,\xi]=0$, $X$, $Y$ being basic. Thus, by
\eqref{ausiliare1},
$[\xi,\nabla^{c}_{X}Y]=\nabla^{c}_{\xi}\nabla^{c}_{X}Y=0$ and hence
$\nabla^{c}_{X}Y$ is basic. Therefore $\nabla^{c}$ locally projects
along the leaves of ${\mathcal F}_{\xi}$ to a linear connection
$\nabla'$ which parallelizes the induced complex and product
structures, since $\nabla^{c}\phi_{\alpha}=0$ for each
$\alpha\in\left\{1,2,3\right\}$. It remains to prove that $\nabla'$
is symmetric. Let $X'$, $Y'$ be any local vector fields on the leaf
space and let $X$, $Y$ be the corresponding basic vector fields such
that $\pi_{\ast}X=X'$ and $\pi_{\ast}Y=Y'$. Then
$T'(X',Y')=\pi_{\ast}T^{c}(X,Y)=\pi_{\ast}(2d\eta(X,Y)\xi)=0$. Thus
$\nabla'$ coincides with the Obata connection.
\end{proof}

\section{The standard bi-paracontact structure of a contact metric $(\kappa,\mu)$-space}

In this section we study one of the main examples of almost
bi-paracontact manifolds, namely we show that any (non-Sasakian)
contact metric $(\kappa,\mu)$-space admits a canonical almost
bi-paracontact structure which satisfies very interesting
properties.

Recall that a contact metric $(\kappa,\mu)$-space is a contact
metric manifold $(M,\phi,\xi,\eta,g)$ such that the Reeb vector
field $\xi$ belongs to the ``$(\kappa,\mu)$-nullity distribution''
i.e.
\begin{equation}\label{eq-km}
R^{g}(X,Y)\xi=\kappa\left(\eta\left(Y\right)X-\eta\left(X\right)Y\right)+\mu\left(\eta\left(Y\right)hX-\eta\left(X\right)hY\right),
\end{equation}
This notion was introduced by Blair, Koufogiorgos and Papantoniou in
\cite{BKP-95}, who proved  the following fundamental results.

\begin{theorem}[\cite{BKP-95}]\label{teoremagreci}
Let $\left(M,\phi,\xi,\eta,g\right)$ be a contact metric
$(\kappa,\mu)$-space. Then necessarily $\kappa\leq 1$.  If
$\kappa=1$ then $h=0$ and $\left(M,\phi,\xi,\eta,g\right)$ is
 Sasakian; if $\kappa<1$, the contact metric structure is not
Sasakian and $M$ admits three mutually orthogonal totally geodesic
distributions ${\mathcal D}(0)=\mathbb{R}\xi$, ${\mathcal
D}_{h}(\lambda)$ and ${\mathcal D}_{h}(-\lambda)=\phi({\mathcal
D}_{h}(\lambda))$ corresponding to the eigenspaces of $h$, where
$\lambda=\sqrt{1-\kappa}$.
\end{theorem}

Furthermore, in \cite{BKP-95} it is proved that any contact metric
$(\kappa,\mu)$-space satisfies \eqref{integrabile1}, hence it is
integrable, and for any $X\in\Gamma({\mathcal D}_{h}(\lambda))$,
$Y\in\Gamma({\mathcal D}_{h}(-\lambda))$,
$\nabla^{g}_{X}Y\in\Gamma({\mathcal
D}_{h}(-\lambda)\oplus\mathbb{R}\xi)$ and
$\nabla^{g}_{Y}X\in\Gamma({\mathcal
D}_{h}(\lambda)\oplus\mathbb{R}\xi)$.

Given a non-Sasakian contact metric $(\kappa,\mu)$-manifold
$(M,\phi,\xi,\eta,g)$, Boeckx \cite{Boeckx-00} proved that the
number $I_{M}:=\frac{1-\frac{\mu}{2}}{\sqrt{1-\kappa}}$, is an
invariant of the contact metric $(\kappa,\mu)$-structure, and he
proved that two non-Sasakian contact metric $(\kappa,\mu)$-manifolds
$(M_1,\phi_1,\xi_1,\eta_1,g_1)$ and $(M_2,\phi_2,\xi_2,\eta_2,g_2)$
are locally isometric as contact metric manifolds if and only if
$I_{M_1}=I_{M_2}$. Then the invariant $I_M$ has been used by Boeckx
for providing a local classification of contact metric
$(\kappa,\mu)$-spaces. An interpretation of the Boeckx invariant in
terms of Legendre foliations is given in \cite{Mino-toap1}.

The standard example of contact metric $(\kappa,\mu)$-manifolds is
given by the tangent sphere bundle $T_{1}N$ of a Riemannian manifold
$N$ of constant curvature $c$ endowed with its standard contact
metric structure. In this case $\kappa=c(2-c)$, $\mu=-2c$ and
$I_{T_{1}N}=\frac{1+c}{|1-c|}$.

The link between contact metric $(\kappa,\mu)$-spaces with the
theory of Legendre foliations was pointed out in
\cite{Mino-Luigia-07} and \cite{Mino-toap1}. In fact any contact
metric $(\kappa,\mu)$-space $(M,\phi,\xi,\eta,g)$ is canonically  a
bi-Legendrian manifold with bi-Legendrian structure given by
$({\mathcal D}_{h}(\lambda),{\mathcal D}_{h}(-\lambda))$, and the
corresponding bi-Legendrian connection preserves the tensors $\phi$,
$h$, $g$ (\cite{Mino-07}, \cite{Mino-Luigia-07}). We prove now that
a contact metric $(\kappa,\mu)$-space admits a further bi-Legendrian
structure which is transverse to $({\mathcal
D}_{h}(\lambda),{\mathcal D}_{h}(-\lambda))$.

\begin{theorem}\label{main1}
In any non-Sasakian contact metric $(\kappa,\mu)$-manifold the
operator $\phi h$ admits three eigenvalues, $0$, of multiplicity
$1$, and $\lambda$, $-\lambda$, each of multiplicity $n$, where
$\lambda:=\sqrt{1-\kappa}$. The corresponding eigendistributions are
given by  ${\mathcal D}_{\phi h}(0)=\mathbb{R}\xi$ and
\begin{align}
{\mathcal D}_{\phi h}(\lambda)&=\left\{X+\phi X |
X\in\Gamma({\mathcal D}_{h}(\lambda)\right\},\label{distribuzione1}\\
{\mathcal D}_{\phi h}(-\lambda)&=\left\{Y+\phi Y |
Y\in\Gamma({\mathcal D}_{h}(-\lambda)\right\}.\label{distribuzione2}
\end{align}
Furthermore, ${\mathcal D}_{\phi h}(\lambda)$ and ${\mathcal
D}_{\phi h}(-\lambda)$  define two mutually orthogonal Legendre
foliations which are transversal to the canonical bi-Legendrian
structure $({\mathcal D}_{h}(\lambda),{\mathcal D}_{h}(-\lambda))$.
\end{theorem}
\begin{proof}
That $\phi h$ admits the eigenvalues $0$ and $\pm\sqrt{1-\kappa}$
follows from the relation $h^2=(\kappa-1)\phi^{2}$ (\cite{BKP-95}).
Since the operator $h$ is symmetric and $\phi$ anti-commutes with
$h$, also $\phi h$ is symmetric and hence it is diagonalizable. Now,
since the kernel of $\phi h$ is generated by the Reeb vector field,
we have that ${\mathcal D}_{\phi h}(0)=\mathbb{R}\xi$. Moreover, if
$X\in\Gamma({\mathcal D}_{\phi h}(\lambda))$, then $\phi h \phi
X=-\phi\phi h X=-\lambda\phi X$, so that $\phi X\in\Gamma({\mathcal
D}_{\phi h}(-\lambda))$. This implies that ${\mathcal D}_{\phi
h}(\lambda)$ and ${\mathcal D}_{\phi h}(-\lambda)$  have equal
dimension $n$, if $2n+1$ is the dimension of $M$.  ${\mathcal
D}_{\phi h}(\lambda)$ and ${\mathcal D}_{\phi h}(-\lambda)$ are in
fact mutually orthogonal. Indeed, for any $X\in\Gamma({\mathcal
D}_{\phi h}(\lambda))$ and $Y\in\Gamma({\mathcal D}_{\phi
h}(-\lambda))$, since the operator $\phi h$ is symmetric, we have
$\lambda g(X,Y)=g(\phi h X,Y)=g(X,\phi h Y)=-\lambda g(X,Y)$, so
that $g(X,Y)=0$. In order to prove \eqref{distribuzione1} first
notice that, for any $X\in\Gamma({\mathcal D}_{h}(\lambda))$, $\phi
h(X+\phi X)=\lambda\phi X-\phi^{2}h X=\lambda(X+\phi X)$ so that
$X+\phi X\in\Gamma({\mathcal D}_{\phi h}(\lambda))$. Thus it remains
to show that, given $Y\in\Gamma({\mathcal D}_{\phi h}(\lambda))$,
there exists $X\in\Gamma({\mathcal D}_{h}(\lambda))$ such that
$Y=X+\phi X$. One can verify that $X:=\frac{1}{2}(Y-\phi Y)$ has the
required properties. In a similar way one proves
\eqref{distribuzione2}. Now we are able to demonstrate the
integrability of the distributions ${\mathcal D}_{\phi h}(\lambda)$
and ${\mathcal D}_{\phi h}(-\lambda)$. Any two sections of
${\mathcal D}_{\phi h}(\lambda)$ can be written as $X+\phi X$ and
$X'+\phi X'$, for some $X,X'\in\Gamma({\mathcal D}_{h}(\lambda))$.
Then, by \eqref{integrabile1}
\begin{align}
\nabla^{g}_{X+\phi X}(X'+\phi X')&=\nabla^{g}_{X}X'+\nabla^{g}_{\phi
X}X'+\phi\nabla^{g}_{X}X'+g(X+hX,X')\xi+\phi\nabla^{g}_{\phi
X}X'\nonumber\\
&\quad+g(\phi X+h\phi X,X')\xi\label{intermedio1}
\end{align}
\begin{align*}
&=\nabla^{g}_{X}X'+\phi\nabla^{g}_{X}X'+\nabla^{g}_{\phi
X}X'+\phi\nabla^{g}_{\phi X}X'+(1+\lambda)g(X,X')\xi.
\end{align*}
Now, $\nabla^{g}_{\phi X}X'\in\Gamma({\mathcal
D}_{h}(\lambda)\oplus\mathbb{R}\xi)$, so that we can decompose
$\nabla^{g}_{\phi X}X'$ along its component tangent to ${\mathcal
D}_{h}(\lambda)$ and the one tangent to $\mathbb{R}\xi$, given by
$\eta(\nabla^{g}_{\phi X}X')\xi=g(\nabla^{g}_{\phi X}X',\xi)\xi$.
But, by \eqref{acca2}, $g(\nabla^{g}_{\phi
X}X',\xi)=-g(X',\nabla^{g}_{\phi X}\xi)=(\lambda-1)g(X,X')$, so that
\eqref{intermedio1} becomes
\begin{align}
\nabla^{g}_{X+\phi X}(X'+\phi
X')&=\nabla^{g}_{X}X'+\phi\nabla^{g}_{X}X' + (\nabla^{g}_{\phi
X}X')_{{\mathcal D}_{h}(\lambda)}+\phi(\nabla^{g}_{\phi
X}X')_{{\mathcal
D}_{h}(\lambda)}\nonumber\\
&\quad + 2\lambda g(X,X')\xi.\label{intermedio2}
\end{align}
Therefore
\begin{align}
[X+\phi X,X'+\phi X']&=[X,X']-\phi[X,X']+(\nabla^{g}_{\phi
X}X')_{{\mathcal D}_{h}(\lambda)}+\phi(\nabla^{g}_{\phi
X}X')_{{\mathcal D}_{h}(\lambda)}\nonumber\\
&\quad - (\nabla^{g}_{\phi X'}X)_{{\mathcal
D}_{h}(\lambda)}+\phi(\nabla^{g}_{\phi X'}X)_{{\mathcal
D}_{h}(\lambda)}.\label{intermedio3}
\end{align}
Due to \eqref{distribuzione1} each of the three terms
$[X,X']-\phi[X,X']$, $(\nabla^{g}_{\phi X}X')_{{\mathcal
D}_{h}(\lambda)}+\phi(\nabla^{g}_{\phi X}X')_{{\mathcal
D}_{h}(\lambda)}$ and $(\nabla^{g}_{\phi X'}X)_{{\mathcal
D}_{h}(\lambda)}+\phi(\nabla^{g}_{\phi X'}X)_{{\mathcal
D}_{h}(\lambda)}$  in the right-hand-side of \eqref{intermedio3} is
a section of ${\mathcal D}_{\phi h}(\lambda)$. Thus we conclude that
${\mathcal D}_{\phi h}(\lambda)$ is involutive. In particular, being
${\mathcal D}_{\phi h}(\lambda)$ an integrable subbundle of
$\mathcal D$, it defines a Legendre foliation of $M$. Analogous
arguments work also for ${\mathcal D}_{\phi h}(-\lambda)$. \ It
remains to prove that ${\mathcal D}_{\phi h}(\lambda)$ and
${\mathcal D}_{\phi h}(-\lambda)$ are transverse to each foliation
of the bi-Legendrian structure $({\mathcal D}_{h}(\lambda),{\mathcal
D}_{h}(-\lambda))$. For instance we show that $TM={\mathcal D}_{\phi
h}(\lambda)\oplus{\mathcal D}_{h}(-\lambda)\oplus\mathbb{R}\xi$, the
other cases being similar. If $X$ is a vector field tangent both to
${\mathcal D}_{\phi h}(\lambda)$ and to ${\mathcal D}_{h}(-\lambda)$
then $\lambda X=\phi h X=-\lambda\phi X$ so that $X=-\phi X$. By
applying $\phi$ we get $X=\phi X$, hence $X=0$. Next, let $Z$ be a
vector field on $M$. Then there exist $X\in\Gamma({\mathcal
D}_{h}(\lambda))$ and $Y\in\Gamma({\mathcal D}_{h}(-\lambda))$ such
that $Z=X+Y+\eta(Z)\xi$. Adding and subtracting $\phi X
\in\Gamma({\mathcal D}_{h}(-\lambda))$ we obtain $Z=(X+\phi
X)+(Y-\phi X) + \eta(Z)\xi$, where $X+\phi X\in\Gamma({\mathcal
D}_{\phi h}(\lambda))$ and $Y-\phi X\in\Gamma({\mathcal
D}_{h}(-\lambda))$.
\end{proof}

Theorem \ref{main1} implies that any (non-Sasakian) contact metric
$(\kappa,\mu)$-space is endowed with two transverse bi-Legendrian
structures $({\mathcal D}_{h}(\lambda),{\mathcal D}_{h}(-\lambda))$
and $({\mathcal D}_{\phi h}(\lambda),{\mathcal D}_{\phi
h}(-\lambda))$ defined by the eigenspaces of the operators $h$ and
$\phi h$ corresponding to the eigenvalues $\pm\lambda$. Thus by
Proposition \ref{legendre2} we conclude that any (non-Sasakian)
contact metric $(\kappa,\mu)$-space $M$ admits an integrable almost
bi-paracontact structure which we call the \emph{standard almost
bi-paracontact structure} of the contact metric $(\kappa,\mu)$-space
$M$. One can easily prove the following result.

\begin{theorem}\label{main2}
Let $(M,\phi,\xi,\eta,g)$ be a non-Sasakian contact metric
$(\kappa,\mu)$-space. The standard almost bi-paracontact structure
of $M$ is given by $(\phi_{1},\phi_{2},\phi_{3})$, where
\begin{equation*}
\phi_{1}:=\frac{1}{\sqrt{1-\kappa}}\phi h, \ \
\phi_{2}:=\frac{1}{\sqrt{1-\kappa}}h, \ \ \phi_{3}:=\phi.
\end{equation*}
\end{theorem}

According to the notation used in $\S$ \ref{biparacontact} we denote
by ${\mathcal D}_{1}^{\pm}$ and ${\mathcal D}_{2}^{\pm}$ the
eigendistributions of $\phi_{1}$ and $\phi_{2}$, respectively,
corresponding to the eigenvalue $\pm 1$. So ${\mathcal
D}_{1}^{\pm}={\mathcal D}_{\phi h}(\pm\lambda)$ and ${\mathcal
D}_{2}^{\pm}={\mathcal D}_{h}(\pm\lambda)$. Then, according to
Theorem \ref{main2}, \eqref{distribuzione1}--\eqref{distribuzione2}
should be compared to Proposition \ref{proprieta2}.

\begin{remark}\label{osservazione}
For each $\alpha\in\left\{1,2\right\}$ we can define a
semi-Riemannian metric $g_{\alpha}$ by setting
\begin{equation}\label{semiriemannian}
g_{\alpha}(X,Y):=d\eta(X,\phi_{\alpha}Y)+\eta(X)\eta(Y)
\end{equation}
for all $X,Y\in\Gamma(TM)$. Then it is easy to check that
$(\phi_{\alpha},\xi,\eta,g_{\alpha})$ is a paracontact metric
structure on $M$. In fact
$(\phi_{\alpha},\xi,\eta,g_{\alpha})=\Psi({\mathcal
D}_{\alpha}^{+},{\mathcal D}_{\alpha}^{-})$ according to the
notation used in $\S$ \ref{preliminari}. Let $\bar\nabla^{pc}$ and
$\nabla^{pc}$ denote the canonical paracontact connections
associated to the paracontact metric structures
$(\phi_{1},\xi,\eta,g_{1})$ and $(\phi_{2},\xi,\eta,g_{2})$,
respectively (cf. Theorem \ref{paratanaka}). Then, since ${\mathcal
D}_{1}^{\pm}$ and ${\mathcal D}_{2}^{\pm}$ are integrable, Theorem
\ref{connection} implies that $\nabla^{pc}=\nabla^{bl}$ and
$\bar\nabla^{pc}=\bar\nabla^{bl}$, where $\nabla^{bl}$ denotes the
bi-Legendrian connection corresponding to the bi-Legendrian
structure $({\mathcal D}_{h}(\lambda),{\mathcal D}_{h}(-\lambda))$
and $\bar\nabla^{bl}$ the bi-Legendrian connection associated to
$({\mathcal D}_{\phi h}(\lambda),{\mathcal D}_{\phi h}(-\lambda))$.
In particular, by \eqref{paratorsion} we have that
\begin{equation}\label{torsioni}
\bar{T}^{bl}(\cdot,\xi)=-\phi_{1}h_{1}, \ \ \
{T}^{bl}(\cdot,\xi)=-\phi_{2}h_{2},
\end{equation}
where $\bar{T}^{bl}$ and $T^{bl}$ denote the torsion tensor fields
of $\bar\nabla^{bl}$ and $\nabla^{bl}$, respectively.
\end{remark}

The bi-Legendrian structure $({\mathcal D}_{2}^{+},{\mathcal
D}_{2}^{-})$ was deeply studied in \cite{Mino-Luigia-07} and
\cite{Mino-toap1}. In the sequel we study the ``new'' bi-Legendrian
structure, $({\mathcal D}_{1}^{+},{\mathcal D}_{1}^{-})$.

\begin{theorem}
The Legendre foliations ${\mathcal D}_{\phi h}(\lambda)$ and
${\mathcal D}_{\phi h}(-\lambda)$ are either non-degenerate or flat.
In particular, ${\mathcal D}_{\phi h}(\lambda)$ and ${\mathcal
D}_{\phi h}(-\lambda)$ are positive definite if and only if $I_M>0$,
negative definite if and only if $I_M<0$, flat if and only if
$I_M=0$.
\end{theorem}
\begin{proof}
Let $X\in\Gamma({\mathcal D}_{\phi h}(\lambda))$. Then the
$(\kappa,\mu)$-nullity condition becomes
\begin{equation}\label{condizione1}
R^{g}(X, \xi)\xi=\kappa X+\mu h X.
\end{equation}
On the other hand,
\begin{align}\label{condizione2}
\nonumber R^{g}(X,\xi)\xi&=-\nabla^{g}_{\xi}\nabla^{g}_{X}\xi-\nabla^{g}_{[X,\xi]}\xi\\
&=-\nabla^{g}_{\xi}\phi X + \nabla^{g}_{\xi}\phi h X + \phi[X,\xi] +
\phi h[X,\xi] \nonumber \\
&=\nabla^{g}_{\phi X}\xi + [\xi,\phi X] + \lambda\nabla^{g}_{X}\xi +
\lambda[\xi,X] + \phi[X,\xi] + \lambda[X,\xi]_{{\mathcal D}_{\phi
h}(\lambda)} - \lambda[X,\xi]_{{\mathcal D}_{\phi h}(-\lambda)}\\
\nonumber &=-\phi^2 X -\phi h \phi X +[\xi,\phi X] + \lambda(-\phi
X-\phi h X) + \lambda[\xi,X]-\phi[\xi,X]-\lambda[\xi,X]_{{\mathcal
D}_{\phi
h}(\lambda)}\\
&\quad+\lambda[\xi,X]_{{\mathcal D}_{\phi h}(-\lambda)} \nonumber\\
&=X+\lambda\phi X + 2h X - \lambda\phi X - \lambda X +
2\lambda[\xi,X]_{{\mathcal D}_{\phi h}(-\lambda)}. \nonumber
\end{align}
Thus \eqref{condizione1} and \eqref{condizione2} imply
\begin{equation*}
\kappa\phi X+\mu\phi h X = (1-\lambda)\phi X+2\phi h
X+2\lambda\phi[\xi,X]_{{\mathcal D}_{\phi h}(-\lambda)},
\end{equation*}
from which it follows that
\begin{equation*}
\phi[\xi,X]_{{\mathcal D}_{\phi
h}(-\lambda)}=\frac{1-\sqrt{1-\kappa}}{2}\phi X -
\frac{1-\frac{\mu}{2}}{\sqrt{1-\kappa}}X=\frac{1-\sqrt{1-\kappa}}{2}\phi
X - I_{M}X.
\end{equation*}
Therefore, by \eqref{invariante3}, we have, for any
$X,X'\in\Gamma({\mathcal D}_{\phi h}(\lambda))$,
\begin{align}\label{invariante1}
\nonumber\Pi_{{\mathcal D}_{\phi
h}(\lambda)}(X,X')&=2g([\xi,X]_{{\mathcal
D}_{\phi h}(-\lambda)},\phi X')\\
&=-2g(\phi[\xi,X]_{{\mathcal D}_{\phi h}(-\lambda)},X')\\
&=-(1-\sqrt{1-\kappa})g(\phi X,X') + 2 I_{M}g(X,X')\nonumber\\
&={2{I_{M}}}g(X,X')\nonumber.
\end{align}
Similarly, one can prove that, for any $Y,Y'\in\Gamma({\mathcal
D}_{\phi h}(-\lambda))$,
\begin{equation}\label{invariante2}
\Pi_{{\mathcal D}_{\phi h}(-\lambda)}(Y,Y')={2{I_{M}}}g(Y,Y').
\end{equation}
The assertion of the theorem then easily follows from the
expressions \eqref{invariante1}, \eqref{invariante2} of the Pang
invariant of the Legendre foliations ${\mathcal D}_{\phi
h}(\lambda)$, ${\mathcal D}_{\phi h}(-\lambda)$.
\end{proof}

Since any (non-Sasakian) contact metric $(\kappa,\mu)$-space
$(M,\phi,\xi,\eta,g)$ is canonically endowed with an almost
bi-paracontact manifold, it admits the linear connections
$\nabla^{1}$, $\nabla^{2}$, $\nabla^{3}$ stated in Theorem
\ref{connessioni} and, moreover, the canonical connection
$\nabla^{c}$ defined in Theorem \ref{connessioneobata}. On the other
hand, to $M$ it is attached also the bi-Legendrian connection
$\nabla^{bl}$ corresponding to the bi-Legendrian structure
$({\mathcal D}_{h}(\lambda),{\mathcal D}_{h}(-\lambda))$, as well as
the bi-Legendrian connection $\bar\nabla^{bl}$ associated with
$({\mathcal D}_{\phi h}(\lambda),{\mathcal D}_{\phi h}(-\lambda))$.
We now find the relations between these connections.

\begin{lemma}
Let $(M,\phi,\xi,\eta,g)$ be a non-Sasakian contact metric
$(\kappa,\mu)$-space and $(\phi_1,\phi_2,\phi_3)$ its standard
almost bi-paracontact structure. Then, for the operators
$h_\alpha:=\frac{1}{2}{\mathcal L}_{\xi}\phi_\alpha$,
$\alpha\in\left\{1,2,3\right\}$, we have
\begin{gather}
h_{1}=-{I_{M}}h=-\left(1-\frac{\mu}{2}\right)\phi_{2},\label{operatore}\\
h_{2}={I_{M}}{\phi
h}+\sqrt{1-\kappa}\phi=\left(1-\frac{\mu}{2}\right)\phi_{1}+\sqrt{1-\kappa}\phi_{3},\label{operatore2}\\
h_{3}=h=\sqrt{1-\kappa}\phi_{2}.\label{operatore3}
\end{gather}
\end{lemma}
\begin{proof}
The proof of \eqref{operatore2} is given in \cite[Lemma
4.5]{Mino-toap2} whereas \eqref{operatore3} is obvious. Then by
using Lemma \ref{lemma1} one can prove \eqref{operatore}.
\end{proof}

Substituting \eqref{operatore}--\eqref{operatore3} in (ii) of
Theorem \ref{connessioni} we get the following corollary.

\begin{corollary}
Let $(M,\phi,\xi,\eta,g)$ be a non-Sasakian contact metric
$(\kappa,\mu)$-space and $(\phi_1,\phi_2,\phi_3)$ its standard
almost bi-paracontact structure. The corresponding connections
$\nabla^{1}$, $\nabla^{2}$, $\nabla^{3}$ stated in Theorem
\ref{connessioni} satisfy the following relations:
\begin{gather}
\nabla^{1}\phi_{1}=0, \ \
\nabla^{1}\phi_{2}=2{\sqrt{1-\kappa}}{\eta\otimes\phi_{3}}, \ \
\nabla^{1}\phi_{3}=2\sqrt{1-\kappa}\eta\otimes\phi_{2},\label{relazione1}\\
\nabla^{2}\phi_{1}=0, \ \ \ \ \nabla^{2}\phi_{2}=0, \ \ \ \
\nabla^{2}\phi_{3}=0,\label{relazione2}\\
\nabla^{3}\phi_{1}=-(2-\mu)\eta\otimes\phi_{2}, \ \
\nabla^{3}\phi_{2}=(2-\mu)\eta\otimes\phi_{1}, \ \
\nabla^{3}\phi_{3}=0.\label{relazione3}
\end{gather}
\end{corollary}

\begin{proposition}\label{connessioni1}
With the notation above, $\nabla^{bl}=\nabla^{2}$ and
$\bar\nabla^{bl}=\nabla^{1}$.
\end{proposition}
\begin{proof}
First notice that $\nabla^{bl}$ satisfies the axioms (i), (ii),
(iii) of Theorem \ref{connessioni} characterizing $\nabla^{2}$.
Indeed by definition $\nabla^{bl}\xi=0$. Next,
$\nabla^{bl}\phi=\nabla^{bl}h=0$ (\cite{Mino-07}) so that, tacking
\eqref{relazione2} into account,
$\nabla^{bl}\phi_{\alpha}=0=\nabla^{2}\phi_{\alpha}$ for each
$\alpha\in\left\{1,2,3\right\}$. Finally, by using the expression
\eqref{torsione} of $T^{bl}$, a direct computation shows that also
(iii) is satisfied. Then $\nabla^{bl}=\nabla^{2}$. \ As second step
we prove that if $S$ denotes the $(1,1)$-type tensor field given by
$S(X,Y):=\nabla^{bl}_{X}Y-\bar\nabla^{bl}_{X}Y$, then we have
\begin{equation}\label{differenza}
S(\cdot,\xi)=0,  \ \ \ \ S(\xi,\cdot)=-\phi h, \ \ \ \ S=0 \
\textrm{on} \ {\mathcal D}.
\end{equation}
Obviously $S(\cdot,\xi)=0$. In order to prove the remaining
relations, let us define a linear connection $\nabla'$ on $M$ by
putting
\begin{equation*}
\nabla'_{E}F:=\left\{
                \begin{array}{ll}
                  \nabla^{bl}_{E}F, & \hbox{for $E\in\Gamma({\mathcal D})$, $F\in\Gamma(TM)$;} \\
                  \bar\nabla^{bl}_{E}F, & \hbox{for $E\in\Gamma(\mathbb{R}\xi)$, $F\in\Gamma(TM)$.}
                \end{array}
              \right.
\end{equation*}
We prove that $\nabla'=\bar\nabla^{bl}$ by checking that $\nabla'$
satisfies the axioms which characterize the bi-Legendrian connection
associated with the bi-Legendrian structure $({\mathcal D}_{\phi
h}(\lambda),{\mathcal D}_{\phi h}(-\lambda))$. First, we prove that
$\nabla'$ preserves the Legendre foliations ${\mathcal D}_{\phi
h}(\lambda)$ and ${\mathcal D}_{\phi h}(-\lambda)$. Due to
\eqref{distribuzione1} any vector field tangent to ${\mathcal
D}_{\phi h}(\lambda)$ has the form $X+\phi X$ for some
$X\in\Gamma({\mathcal D}_{h}(\lambda))$. Then, for any
$Z\in\Gamma({\mathcal D})$, we have $\nabla'_{Z}(X+\phi
X)=\nabla'_{Z}X+\nabla'_{Z}\phi X=\nabla'_{Z}X+\nabla^{bl}_{Z}\phi
X=\nabla'_{Z}X+\phi\nabla^{bl}_{Z}X=\nabla'_{Z}X+\phi\nabla'_{Z}X$.
Since $\nabla'_{Z}X=\nabla^{bl}_{Z}X\in\Gamma({\mathcal
D}_{h}(\lambda))$, we conclude that $\nabla'_{Z}(X+\phi
X)\in\Gamma({\mathcal D}_{\phi h}(\lambda))$. Thus
$\nabla'_{Z}{\mathcal D}_{\phi h}(\lambda)\subset{\mathcal D}_{\phi
h}(\lambda)$. Moreover, $\nabla'_{\xi}{\mathcal D}_{\phi
h}(\lambda)=\bar\nabla_{\xi}{\mathcal D}_{\phi
h}(\lambda)\subset{\mathcal D}_{\phi h}(\lambda)$. Analogously one
can prove that $\nabla'$ preserves ${\mathcal D}_{\phi
h}(-\lambda)$. Next, $\nabla'd\eta=0$ since $\nabla^{bl}d\eta=0$ and
$\bar\nabla^{bl}d\eta=0$. Finally, one can easily prove that
$T'(Z,\xi)=\bar{T}^{bl}(Z,\xi)=[\xi,Z_{{\mathcal D}_{\phi
h}(\lambda)}]_{{\mathcal D}_{\phi h}(-\lambda)}+[\xi,Z_{{\mathcal
D}_{\phi h}(-\lambda)}]_{{\mathcal D}_{\phi h}(\lambda)}$ and
$T'(Z,Z')=T^{bl}(Z,Z')=2d\eta(Z,Z')\xi$ for any
$Z,Z'\in\Gamma({\mathcal D})$. Thus, by Theorem \ref{biconnection},
$\nabla'=\bar\nabla^{bl}$ and hence
 $S=0$ on $\mathcal D$. Finally, by  \eqref{torsioni}
\begin{equation*}
\nabla^{bl}_{\xi}Z=\nabla^{bl}_{Z}\xi-T^{bl}(Z,\xi)-[Z,\xi]=\phi_{2}h_{2}Z+[\xi,Z]
\end{equation*}
and, analogously,
\begin{equation*}
\bar\nabla^{bl}_{\xi}Z=\phi_{1}h_{1}Z+[\xi,Z].
\end{equation*}
Therefore, by using \eqref{operatore} and \eqref{operatore2}, one
finds $S(\xi,Z)=\phi_{2}h_{2}Z-\phi_{1}h_{1}Z=-\phi h Z$. Thus
\eqref{differenza} is completely proved. In particular, one obtains
\begin{equation}\label{formulacurvatura4}
\bar\nabla^{bl}_{\xi}\phi=\nabla^{bl}_{\xi}\phi+\phi h\phi-\phi^2
h=2h
\end{equation}
and
\begin{equation}\label{formulacurvatura4bis}
\bar\nabla^{bl}_{\xi}h=\nabla^{bl}_{\xi}h+\phi h^2-h\phi h=2\phi
h^2=2(1-\kappa)\phi.
\end{equation}
Then $\bar\nabla^{bl}$ satisfies \eqref{relazione1}. Since it easily
satisfies also the other two conditions which uniquely define the
connection $\nabla^{1}$, we conclude that
$\bar\nabla^{bl}=\nabla^{1}$.
\end{proof}

The paracontact metric structure $(\phi_{2},\xi,\eta,g_{2})$ defined
in Remark \ref{osservazione} was studied  in \cite{Mino-toap2}. Now
we are able to study $(\phi_{1},\xi,\eta,g_{1})$. We show that both
the paracontact metric structures satisfy a nullity condition.

\begin{theorem}\label{indotte}
Let $(M,\phi,\xi,\eta,g)$ be a non-Sasakian contact metric
$(\kappa,\mu)$-space and let $(\phi_{1},\phi_{2},\phi_{3})$ be its
standard almost bi-paracontact structure. Let $g_1$ and $g_2$ denote
the semi-Riemannian metrics defined by \eqref{semiriemannian},
compatible with the almost paracontact structures $\phi_1$ and
$\phi_2$, respectively. Then the paracontact metric structures
$(\phi_\alpha,\xi,\eta,g_\alpha)$, $\alpha\in\left\{1,2\right\}$,
satisfy
\begin{equation*}
R^{g_\alpha}(X,Y)\xi=\kappa_{\alpha}(\eta(Y)X-\eta(X)Y)+\mu_{\alpha}(\eta(Y)h_{\alpha}X-\eta(X)h_{\alpha}Y)
\end{equation*}
where
\begin{gather}
\kappa_{1}=\left(1-\frac{\mu}{2}\right)^2-1, \ \ \
\mu_{1}=2(1-\sqrt{1-\kappa}),\label{valori01}\\
\kappa_{2}=\kappa-2+\left(1-\frac{\mu}{2}\right)^2, \ \ \
\mu_{2}=2.\label{valori02}
\end{gather}
Furthermore, $I_{M}=0$ if and only if $(\phi_{1},\xi,\eta,g_{1})$ is
para-Sasakian.
\end{theorem}
\begin{proof}
For the case $\alpha=2$ the assertion was already proved in
\cite{Mino-toap2}. We prove the case $\alpha=1$. First notice that,
as ${\mathcal D}_{1}^{+}$ and ${\mathcal D}_{1}^{-}$ are involutive,
the paracontact metric structure $(\phi_{1},\xi,\eta,g_{1})$
satisfies \eqref{integrabile2} (cf. \cite{zamkovoy}). Then by
\eqref{paradefinition} we have that
\begin{align}\label{formulacurvatura2}
\nonumber(\nabla^{g_{1}}_{X}h_{1})Y&=(\bar\nabla^{pc}_{X}h_{1})Y-2\eta(X)\phi_{1}h_{1}Y-\eta(Y)\phi_{1}h_{1}X+\eta(Y)\phi_{1}h_{1}^{2}X-g_{1}(X,\phi_{1}h_{1}Y)\xi\\
&\quad+g_{1}(h_{1}X,\phi_{1}h_{1}Y)\xi.
\end{align}
Moreover, due to \eqref{operatore} and Proposition
\ref{connessioni1} we get
\begin{equation}\label{formulacurvatura5}
(\bar\nabla^{pc}_{X}h_{1})Y=(\bar\nabla^{bl}_{X}h_{1})Y=(\nabla^{1}_{X}h_{1})Y=-\left(1-\frac{\mu}{2}\right)(\nabla^{1}_{X}\phi_{2})Y=(\mu-2)\sqrt{1-\kappa}\eta(X)\phi_{3}Y.
\end{equation}
Thus, by replacing \eqref{integrabile2}, \eqref{formulacurvatura2}
and \eqref{formulacurvatura5}  in \eqref{formulacurvatura1} we find
\begin{align*}
R^{g_{1}}(X,Y)\xi&=-\eta(Y)(X-h_{1}X)+g_{1}(X-h_{1}X,Y)\xi+\eta(X)(Y-h_{1}Y)-g_{1}(Y-h_{1}Y,X)\xi\\
&\quad-g_{1}(X-h_{1}X,h_{1}Y)\xi+\phi_{1}((\bar\nabla^{pc}_{X}h_{1})Y)-2\eta(X)\phi_{1}^{2}h_{1}Y+\eta(Y)\phi_{1}h_{1}\phi_{1}X\\
&\quad+\eta(Y)\phi_{1}^{2}h_{1}X+g_{1}(Y-h_{1}Y,h_{1}X)\xi-\phi_{1}((\bar\nabla^{pc}_{Y}h_{1})X)+2\eta(Y)\phi_{1}^{2}h_{1}X\\
&\quad-\eta(X)\phi_{1}h_{1}\phi_{1}Y-\eta(X)\phi_{1}^{2}h_{1}Y\\
&=-\eta(Y)X+\eta(X)Y+(\mu-2)\sqrt{1-\kappa}\eta(X)\phi_{1}\phi_{3}Y-2\eta(X)\phi_{1}^{2}h_{1}Y+\eta(Y)\phi_{1}^{2}h_{1}^{2}X\\
&\quad-(\mu-2)\sqrt{1-\kappa}\eta(Y)\phi_{1}\phi_{3}X+2\eta(Y)\phi_{1}^{2}h_{1}X-\eta(X)\phi_{1}^{2}h_{1}^{2}Y
\end{align*}
\begin{align*}
&=-\eta(Y)X+\eta(X)Y+(\mu-2)\sqrt{1-\kappa}\eta(X)\phi_{2}Y-2\eta(X)h_{1}Y-\left(1-\frac{\mu}{2}\right)^2\eta(Y)\phi_{2}^{2}X\\
&\quad-(\mu-2)\sqrt{1-\kappa}\eta(Y)\phi_{2}X+2\eta(Y)h_{1}X+\left(1-\frac{\mu}{2}\right)^{2}\eta(X)\phi_{2}^{2}Y\\
&=-\eta(Y)X+\eta(X)Y+2\sqrt{1-\kappa}\eta(X)h_{1}Y-2\eta(X)h_{1}Y+\left(1-\frac{\mu}{2}\right)^{2}\eta(Y)X\\
&\quad-2\sqrt{1-\kappa}\eta(Y)h_{1}X+2\eta(Y)h_{1}X-\left(1-\frac{\mu}{2}\right)^{2}\eta(X)Y\\
&=\left(\left(1-\frac{\mu}{2}\right)^{2}-1\right)(\eta(Y)X-\eta(X)Y)+2(1-\sqrt{1-\kappa})(\eta(Y)h_{1}X-\eta(X)h_{1}Y).
\end{align*}
For the last assertion in the statement of the theorem, we have that
$I_{M}=0$ if and only if $\mu=2$, i.e., by \eqref{operatore}, if and
only if $h_{1}=0$. As the paracontact metric structure
$(\phi_{1},\xi,\eta,g_{1})$ is integrable, the assert follows from
Corollary \ref{paracontatto4}.
\end{proof}

We now study the special properties of the connection $\nabla^{c}$
(cf. Theorem \ref{connessioneobata}) associated to the standard
almost bi-paracontact structure $(\phi_1,\phi_2,\phi_3)$ of a
(non-Sasakian) contact metric $(\kappa,\mu)$-space
$(M,\phi,\xi,\eta,g)$. We call $\nabla^{c}$ the \emph{ canonical
connection of the contact metric $(\kappa,\mu)$-space $M$}.

\begin{lemma}
The torsion tensor field of the canonical connection of a
non-Sasakian contact metric $(\kappa,\mu)$-space
$(M,\phi,\xi,\eta,g)$ is given by
\begin{align}
T^{c}(X,Y)&=\frac{2}{3}\left(\eta(Y)\left(\left(1-\frac{\mu}{2}\right)\phi
X+\phi h X\right)-\eta(X)\left(\left(1-\frac{\mu}{2}\right)\phi
Y+\phi h Y\right)\right)\nonumber\\
&\quad+2d\eta(X,Y)\xi.\label{torsione6}
\end{align}
In particular,
\begin{equation}\label{torsioneobata1}
T^{c}(X,\xi)=\frac{2}{3}\left(\left(1-\frac{\mu}{2}\right)\phi
X+\phi h X\right).
\end{equation}
\end{lemma}
\begin{proof}
First of all notice that, being the almost bi-paracontact structure
$(\phi_{1},\phi_{2},\phi_{3})$ integrable, \eqref{legendrian} holds.
Then by replacing \eqref{contatto3}, \eqref{paracontatto3},
\eqref{legendrian} into (iii) of Theorem \ref{connessioneobata} we
obtain
\begin{align}\
T^{c}(X,Y)&=2d\eta(X,Y)\xi+\frac{1}{6}\bigl(-2\eta(Y)\phi_{1}h_{1}X+2\eta(X)\phi_{1}h_{1}Y-2\eta(Y)\phi_{2}h_{2}X\nonumber\\
&\quad+2\eta(X)\phi_{2}h_{2}Y+2\eta(Y)\phi_{3}h_{3}X-2\eta(X)\phi_{3}h_{3}Y\bigr).\label{torsione5}
\end{align}
By substituting \eqref{operatore} and \eqref{operatore2} in
\eqref{torsione5}, a straightforward computation yields
\eqref{torsione6}.
\end{proof}

\begin{proposition}\label{connessioni2}
With the notation above, we have for any $X,Y\in\Gamma({\mathcal
D})$,
\begin{equation*}
\nabla^{c}_{X}Y=\nabla^{1}_{X}Y=\nabla^{2}_{X}Y=\nabla^{3}_{X}Y.
\end{equation*}
\end{proposition}
\begin{proof}
Let $\nabla'$ be the linear connection defined by
\begin{equation*}
\nabla'_{E}F:=\left\{
                \begin{array}{ll}
                  \nabla^{bl}_{E}F, & \hbox{if $E\in\Gamma({\mathcal D})$;} \\
                  \nabla^{c}_{E}F, & \hbox{if $E\in\Gamma(\mathbb{R}\xi)$.}
                \end{array}
              \right.
\end{equation*}
We check that $\nabla'$ satisfies (i), (ii), (iii) of Theorem
\ref{connessioneobata}. First of all, obviously $\nabla'\xi=0$.
Next, for all $X,Y\in\Gamma({\mathcal D})$, by \eqref{torsione6},
$T'(X,Y)=T^{bl}(X,Y)=2d\eta(X,Y)\xi=T^{c}(X,Y)$ and
$T'(X,\xi)=T^{c}(X,\xi)$. Finally, for all $X,Y\in\Gamma({\mathcal
D})$, we have
$(\nabla'_{X}\phi_\alpha)Y=(\nabla^{bl}_{X}\phi_\alpha)Y=0=(\nabla^{c}_{X}\phi_{\alpha})Y$
for each $\alpha\in\left\{1,2,3\right\}$, since
$\nabla^{bl}\phi=\nabla^{bl}h=0$. Moreover, by definition,
$(\nabla'_{\xi}\phi_{\alpha})X=(\nabla^{c}_{\xi}\phi_{\alpha})X$.
Thus by the uniqueness of $\nabla^{c}$ we have that
$\nabla'=\nabla^{c}$. Then, since by Proposition \ref{connessioni1}
$\nabla^{bl}=\nabla^{2}$, we have that $\nabla^{2}$ and $\nabla^{c}$
coincide on the contact distribution. Moreover, Proposition
\ref{connessioni1} and \eqref{differenza} imply that also
$\nabla^{1}=\bar\nabla^{bl}$ and $\nabla^{c}$ coincide on $\mathcal
D$. The same property is then necessarily satisfied by $\nabla^{3}$
since $\nabla^{c}$ is the barycenter of $\nabla^{1}$, $\nabla^{2}$,
$\nabla^{3}$.
\end{proof}

\begin{corollary}
The canonical connection $\nabla^{c}$ of a contact metric
$(\kappa,\mu)$-space $(M,\phi,\xi,\eta,g)$ is a \emph{contact
connection}, i.e. $\nabla^{c}\eta=\nabla^{c}d\eta=0$, and satisfies
\begin{align}
\nabla^{c}\phi_{1}&=\quad\quad\quad\quad -\frac{2}{3}\left(1-\frac{\mu}{2}\right)\eta\otimes\phi_{2}\label{condizioniobata01}\\
\nabla^{c}\phi_{2}&=\frac{2}{3}\left(1-\frac{\mu}{2}\right)\eta\otimes\phi_{1} \quad\quad +\frac{2}{3}\sqrt{1-\kappa}\eta\otimes\phi_{3}\label{condizioniobata02}\\
\nabla^{c}\phi_{3}&=\quad\quad\quad\quad\quad
\frac{2}{3}\sqrt{1-\kappa}\eta\otimes\phi_{2}\label{condizioniobata03}
\end{align}
\end{corollary}
\begin{proof}
By Proposition \ref{connessioni1} and Proposition \ref{connessioni2}
we have, for all $X,Y,Z\in\Gamma({\mathcal D})$,
$(\nabla^{c}_{X}d\eta)(Y,Z)=(\nabla^{2}_{X}d\eta)(Y,Z)=(\nabla^{bl}_{X}d\eta)(Y,Z)=0$
and, since $\nabla^{c}\xi=0$, $(\nabla^{c}_{X}d\eta)(Y,\xi)=0$.
Moreover, from \eqref{torsioneobata1} it follows that
\begin{equation}\label{formulaobata1}
\nabla_{\xi}^{c}X=[\xi,X]-\frac{2}{3}\left(\left(1-\frac{\mu}{2}\right)\phi
X+\phi h X\right).
\end{equation}
Then \eqref{formulaobata1} yields
\begin{align*}
(\nabla^{c}_{\xi}d\eta)(X,Y)&=\xi(d\eta(X,Y))-d\eta([\xi,X],Y)+\frac{2}{3}\left(1-\frac{\mu}{2}\right)d\eta(\phi
X,Y)+\frac{2}{3}d\eta(\phi h X,Y)\\
&\quad-d\eta(X,[\xi,Y])+\frac{2}{3}\left(1-\frac{\mu}{2}\right)d\eta(X,\phi
Y)+\frac{2}{3}d\eta(X,\phi h Y)\\
&=({\mathcal L}_{\xi}d\eta)(X,Y)+\frac{2}{3}g(\phi h X,\phi
Y)+\frac{2}{3}g(X,\phi^{2}h Y)=0,
\end{align*}
since ${\mathcal L}_{\xi}d\eta=0$ and $h$ is a symmetric operator.
Finally, \eqref{condizioniobata01}--\eqref{condizioniobata03} follow
from (ii) of Theorem \ref{connessioneobata} and from
\eqref{operatore}, \eqref{operatore2}.
\end{proof}


Conversely, we show that
\eqref{condizioniobata01}--\eqref{condizioniobata03} in some sense
characterize the existence of a  contact metric
$(\kappa,\mu)$-structure on an almost bi-paracontact manifold.

\begin{theorem}\label{main3}
Let $(\phi_{1},\phi_{2},\phi_{3})$ be an integrable almost
bi-paracontact structure on the contact manifold $(M,\eta)$ such
that the associated canonical connection satisfies
$\nabla^{c}d\eta=0$ and
\begin{align}
\nabla^{c}\phi_{1}&=\quad\quad -a\eta\otimes\phi_{2}\label{condizioniobata1}\\
\nabla^{c}\phi_{2}&=a\eta\otimes\phi_{1} \quad\quad +b\eta\otimes\phi_{3}\label{condizioniobata2}\\
\nabla^{c}\phi_{3}&=\quad\quad\quad
b\eta\otimes\phi_{2}\label{condizioniobata3}
\end{align}
for some $a>0$ (respectively, $a<0$) and $b>0$. Let us define
\begin{equation}\label{definizionemetrica}
g_{1}:=d\eta(\cdot,\phi_{1}\cdot)+\eta\otimes\eta, \ \
g_{2}:=d\eta(\cdot,\phi_{2}\cdot)+\eta\otimes\eta, \ \
g_{3}:=-d\eta(\cdot,\phi_{3}\cdot)+\eta\otimes\eta
\end{equation}
and assume that the symmetric bilinear form
$\pi_{1}:=g_{1}(h_{1}\cdot,\cdot)$ is positive definite
(respectively, negative definite). Then, for each
$\alpha\in\left\{1,2\right\}$, $(\phi_{\alpha},\xi,\eta,g_{\alpha})$
is a paracontact metric $(\kappa_{\alpha},\mu_{\alpha})$-structure
and $(\phi_{3},\xi,\eta,g_{3})$ is a contact metric
$(\kappa_{3},\mu_{3})$-structure, where
\begin{gather}
\kappa_{1}:=\frac{9}{4}a^{2}-1, \ \ \
\mu_{1}:=2-3b, \label{valori1} \\
\kappa_{2}:=\frac{9}{4}\left(a^{2}-b^{2}\right)-1, \ \ \
\mu_{2}:=2,\label{valori2} \\
\kappa_{3}:=1-\frac{9}{4}b^{2}, \ \ \ \mu_{3}:=2+3a. \label{valori3}
\end{gather}
Moreover, $(\phi_1,\phi_2,\phi_3)$ is the standard almost
bi-paracontact structure of the contact metric
$(\kappa_{3},\mu_{3})$-manifold $(M,\phi_{3},\xi,\eta,g_{3})$.
\end{theorem}
\begin{proof}
Since the almost bi-paracontact structure is assumed to be
integrable, we have in particular, by Proposition \ref{proprieta4},
that the bilinear forms $g_{1}$, $g_{2}$, $g_{3}$, defined by
\eqref{definizionemetrica}, are symmetric, so that the definition is
well posed. Notice that, by construction, for each
$\alpha\in\left\{1,2,3\right\}$, $g_{\alpha}$ is compatible with the
corresponding structure, i.e.
\begin{equation*}
g_{\alpha}(\phi_{\alpha}X,\phi_{\alpha}Y)=-\epsilon\left(
g_{\alpha}(X,Y)-\eta(X)\eta(Y)\right)
\end{equation*}
where we have posed $\epsilon=1$ if $\alpha\in\left\{1,2\right\}$
and $\epsilon=-1$ if $\alpha=3$. Moreover, each $g_{\alpha}$ is, by
definition, an associated metric, i.e.
$d\eta(X,Y)=g_{\alpha}(X,\phi_{\alpha}Y)$ for all
$X,Y\in\Gamma(TM)$. Furthermore, by comparing
\eqref{condizioniobata1}--\eqref{condizioniobata3} with (ii) of
Theorem \ref{connessioneobata} we have that
\begin{equation}\label{confronto}
h_{1}=-\frac{3}{2}a\phi_{2}, \ \ \
h_{2}=\frac{3}{2}\left(a\phi_{1}+b\phi_{3}\right), \ \ \
h_{3}=\frac{3}{2}b\phi_{2}.
\end{equation}
Hence, by \eqref{confronto}, we have, for all $X,Y\in\Gamma(TM)$,
\begin{align*}
g_{3}(X,Y)&=-d\eta(X,\phi_{3}Y)+\eta(X)\eta(Y)\\
&=-g_{1}(X,\phi_{1}\phi_{3}Y)+\eta(X)\eta(Y)\\
&=-g_{1}(X,\phi_{2}Y)+\eta(X)\eta(Y)\\
&=\frac{2}{3a} g_{1}(X,h_{1}Y)\\
&=\frac{2}{3a}\pi_{1}(X,Y).
\end{align*}
Then the assumptions of positive definiteness of $\pi_1$ and $a>0$
imply that $g_{3}$ is a Riemannian metric. It follows that
$(\phi_{\alpha},\xi,\eta,g_{\alpha})$ is a paracontact metric
structure for $\alpha\in\left\{1,2\right\}$ and a contact metric
structure for $\alpha=3$. Now, since the almost bi-paracontact
structure $(\phi_1,\phi_2,\phi_3)$ is integrable, by Corollary
\ref{proprieta7}, the tensor fields $N^{(1)}_{\phi_1}$,
$N^{(1)}_{\phi_2}$, $N^{(1)}_{\phi_3}$ vanish on $\mathcal D$.
Moreover, Proposition \ref{proprieta4} implies that
$d(\phi_{1}X,\phi_{1}Y)=d(\phi_{2}X,\phi_{2}Y)=-d\eta(\phi_{3}X,\phi_{3}Y)=-d\eta(X,Y)$
for any $X,Y\in\Gamma({\mathcal D})$. Hence, taking (iii) of Theorem
\ref{connessioneobata} into account, the torsion of the canonical
connection is given by
\begin{equation}\label{torsione4}
T^{c}(X,Y)=2d\eta(X,Y)\xi
\end{equation}
for all $X,Y\in\Gamma({\mathcal D})$. We now are able to prove that
on the contact distribution the canonical connection and the Levi
Civita connection of $g_{3}$ are related by the formula
\begin{equation}\label{relazioneconnessioni}
\nabla^{c}_{X}Y=\nabla^{g_{3}}_{X}Y-\eta(\nabla^{g_{3}}_{X}Y)\xi.
\end{equation}
Indeed, let us define a linear connection $\nabla'$ on $M$ by
\begin{equation*}
\nabla'_{X}Y:=\left\{
                \begin{array}{ll}
                  \nabla^{c}_{X}Y+\eta(\nabla^{g_{3}}_{X}Y)\xi, & \hbox{if $X,Y\in\Gamma({\mathcal D})$;} \\
                  \nabla^{g_3}_{X}Y, & \hbox{elsewhere.}
                \end{array}
              \right.
\end{equation*}
We prove that in fact $\nabla'$ coincides with the Levi Civita
connection of $(M,g_3)$. For any $X,Y,Z\in\Gamma({\mathcal D})$ we
have
\begin{align*}
(\nabla'_{X}g_{3})(Y,Z)&=(\nabla^{c}_{X}g_{3})(Y,Z)-\eta(Z)\eta(\nabla^{g_{3}}_{X}Y)-\eta(Y)\eta(\nabla^{g_3}_{X}Z)\\
&=-X(d\eta(Y,\phi_{3}Z))+d\eta(\nabla^{c}_{X}Y,\phi_{3}Z)+d\eta(Y,\phi_{3}\nabla^{c}_{X}Z)\\
&=-X(d\eta(Y,\phi_{3}Z))+d\eta(\nabla^{c}_{X}Y,\phi_{3}Z)+d\eta(Y,\nabla^{c}_{X}\phi_{3}Z)\\
&=-(\nabla^{c}_{X}d\eta)(Y,\phi_{3}Z)=0,
\end{align*}
\begin{equation*}
(\nabla'_{X}g_{3})(Y,\xi)=(\nabla^{g_{3}}_{X}g_{3})(Y,\xi)-\eta(\nabla^{c}_{X}Y)=0
\end{equation*}
and
\begin{equation*}
(\nabla'_{\xi}g_{3})(Y,Z)=(\nabla^{g_{3}}_{\xi}g_{3})(Y,Z)=0.
\end{equation*}
Next, by \eqref{torsione4}
\begin{equation*}
T'(X,Y)=T^{c}(X,Y)+\eta(\nabla^{g_3}_{X}Y)\xi-\eta(\nabla^{g_3}_{Y}X)\xi=2d\eta(X,Y)\xi+\eta([X,Y])\xi=0,
\end{equation*}
and $T'(X,\xi)=T^{g_3}(X,\xi)=0$. Thus $\nabla'=\nabla^{g_{3}}$ and
\eqref{relazioneconnessioni} follows. Then \eqref{condizioniobata2},
\eqref{confronto} and \eqref{relazioneconnessioni}  yield, for any
$X,Y,Z\in\Gamma({\mathcal D})$,
\begin{align*}
g_{3}((\nabla^{g_{3}}_{X}h_{3})Y,Z)&=g_{3}((\nabla^{c}_{X}h_{3})Y,Z)+\eta(\nabla^{g_{3}}_{X}h_{3}Y)\eta(Z)\\
&=\frac{3}{2}b g_{3}((\nabla^{c}_{X}\phi_{2})Y,Z)\\
&=\frac{3}{2}ab\eta(X)g_{3}(\phi_{1}X,Z)+\frac{3}{2}b^2\eta(X)g_{3}(\phi_{3}X,Z)=0.
\end{align*}
Therefore the tensor field $h_{3}$ is ``$\eta$-parallel'' (cf.
\cite{boeckx2}) and so, by \cite[Theorem 4]{boeckx2},
$(\phi_{3},\xi,\eta,g_{3})$ is a contact metric
$\left(\kappa,\mu\right)$-space. The values of $\kappa$ and $\mu$
can be found by comparing
\eqref{condizioniobata1}--\eqref{condizioniobata3} with
\eqref{condizioniobata01}--\eqref{condizioniobata03}. After a
straightforward computation it turns out that they are given by
\eqref{valori3}. The remaining part of the theorem follows from
Theorem \ref{indotte}. In particular, \eqref{valori1} and
\eqref{valori2} are consequence of \eqref{valori01} and
\eqref{valori02}, respectively. \ The case $a<0$ can be proved in a
similar way.
\end{proof}

Formulae \eqref{operatore}--\eqref{operatore3} together with (a) of
Lemma \ref{lemma1} allow us to define a supplementary almost
bi-paracontact structure on a non-Sasakian contact metric
$(\kappa,\mu)$-space. In fact, by \eqref{operatore2} we have
\begin{align}
h_{2}^{2}&=\left(1-\frac{\mu}{2}\right)^{2}\phi_{1}^{2}+\left(1-\frac{\mu}{2}\right)\sqrt{1-\kappa}\phi_{1}\phi_{3}+\left(1-\frac{\mu}{2}\right)\sqrt{1-\kappa}\phi_{3}\phi_{1}+(1-\kappa)\phi_{3}^{2}\nonumber\\
&=\left(\left(1-\frac{\mu}{2}\right)^{2}-\left(1-\kappa\right)\right)\left(I-\eta\otimes\xi\right).\label{operatorenuovo}
\end{align}
Therefore, under the assumption that
$\left(1-\frac{\mu}{2}\right)^2\neq 1-\kappa$, we are led to
consider the tensor field
\begin{align}
\label{operatorenuovobis}\psi:=&\frac{1}{\sqrt{\left|\left(1-\frac{\mu}{2}\right)^2-\left(1-\kappa\right)\right|}}h_{2}=\frac{1}{\sqrt{\left|\left(1-\frac{\mu}{2}\right)^2-\left(1-\kappa\right)\right|}}\left(\left(1-\frac{\mu}{2}\right)\phi_{1}+\sqrt{1-\kappa}\phi_{3}\right)
\end{align}
By \eqref{operatorenuovo} we see that if
$\left(1-\frac{\mu}{2}\right)^2-\left(1-\kappa\right)>0$ then the
tensor field $\psi$ satisfies $\psi^{2}=I-\eta\otimes\xi$, whereas
if $\left(1-\frac{\mu}{2}\right)^2-\left(1-\kappa\right)<0$ we have
$\psi^{2}=-I+\eta\otimes\xi$. Notice that
$\left(1-\frac{\mu}{2}\right)^2-\left(1-\kappa\right)>0$ if and only
if  $|I_M|>0$. Therefore we are able to prove the following theorem.

\begin{theorem}\label{main4}
Let $(M,\phi,\xi,\eta,g)$ be a non-Sasakian contact metric
$(\kappa,\mu)$-space such that $I_{M}\neq\pm 1$.
\begin{enumerate}
  \item[(i)] If $|I_{M}|>1$ then $M$ admits an integrable almost
bi-paracontact structure $(\phi'_{1},\phi'_{2},\phi'_{3})$, given by
\begin{align*}
\phi'_{1}:=&\frac{1}{\sqrt{\left(1-\frac{\mu}{2}\right)^2-\left(1-\kappa\right)}}\left(I_{M}\phi
h+\sqrt{1-\kappa}\phi\right)\\
\phi'_{2}:=&\frac{1}{\sqrt{1-\kappa}}h\\
\phi'_{3}:=&\frac{1}{\sqrt{\left(1-\frac{\mu}{2}\right)^2-\left(1-\kappa\right)}}\left(I_{M}h+\sqrt{1-\kappa}\phi
h\right).
\end{align*}
  \item[(ii)]If $|I_{M}|<1$ then $M$ admits an integrable almost bi-paracontact structure $(\phi''_{1},\phi''_{2},\phi''_{3})$, given by
\begin{align*}
\phi''_{1}:=&\frac{1}{\sqrt{1-\kappa}}h\\
\phi''_{2}:=&\frac{1}{\sqrt{1-\kappa-\left(1-\frac{\mu}{2}\right)^2}}\left(I_{M}h+\sqrt{1-\kappa}\phi
h\right)\\
\phi''_{3}:=&\frac{1}{\sqrt{1-\kappa-\left(1-\frac{\mu}{2}\right)^2}}\left(I_{M}\phi
h+\sqrt{1-\kappa}\phi\right).
\end{align*}
\end{enumerate}
\end{theorem}
\begin{proof}
Let us assume $|I_{M}|>1$. In order to relieve the notation, we put
$\alpha:=1-\frac{\mu}{2}$ and $\beta:=\sqrt{1-\kappa}$.   As
remarked before, by a direct computation one proves that
${\phi'}_{1}^{2}=I-\eta\otimes\xi$. Moreover, by (a) of  Lemma
\ref{lemma1}, $\phi_{2}h_{2}=-h_{2}\phi_{2}$, so that
${\phi'}_{1}=\frac{1}{\sqrt{\alpha^{2}-\beta^{2}}}h_2$ and
${\phi'}_{2}=\phi_2$ anti-commute. Thus
$({\phi'}_{1},{\phi'}_{2},{\phi'}_{3}={\phi'}_{1}{\phi'}_{2})$ is an
almost bi-paracontact structure on $(M,\eta)$. We prove that it is
integrable, by showing that the eigendistributions ${{\mathcal
D}'}_{1}^{\pm}$ associated to ${\phi'}_1$ define Legendre
foliations, since we already know that ${{\mathcal
D}'}_{2}^{\pm}={{\mathcal D}}_{2}^{\pm}$ do. First we show that
${{\mathcal D}'}_{1}^{+}$ is a Legendrian distribution. For any
$X,X'\in\Gamma({{\mathcal D}'}_{1}^{+})$ we have
\begin{align}
\nonumber d\eta(X,X')&=d\eta({\phi'}_{1}X,{\phi'}_{1}X')\\
&=\frac{1}{\alpha^{2}-\beta^{2}}\bigl(\alpha^{2}d\eta(\phi_{1}X,\phi_{1}X')+\alpha\beta
d\eta(\phi_{1}X,\phi_{3}X')+\alpha\beta
d\eta(\phi_{3}X,\phi_{1}X') \label{intermedio4}\\
\nonumber &\quad+\beta^{2}d\eta(\phi_{3}X,\phi_{3}X')\bigr).
\end{align}
Now, notice that \
$d\eta(\phi_{1}X,\phi_{1}X')=-d\eta(\phi_{3}X,\phi_{3}X')=-d\eta(X,X')$,
 and
$d\eta(\phi_{1}X,\phi_{3}X')=d\eta(\phi_{1}X,\phi_{1}\phi_{2}X')=-d\eta(X,\phi_{2}X')=-d\eta(\phi_{3}X,\phi_{1}X')$,
so that \eqref{intermedio4} becomes
\begin{equation*}
d\eta(X,X')=-\frac{\alpha^{2}-\beta^{2}}{\sqrt{\alpha^{2}-\beta^{2}}}d\eta(X,X')=-\sqrt{\alpha^{2}-\beta^{2}}d\eta(X,X').
\end{equation*}
Hence $d\eta(X,X')=0$. It remains to prove that ${{\mathcal
D}'}_{1}^{+}$ is involutive. Take $X,X'\in\Gamma({{\mathcal
D}'}_{1}^{+})$. By \eqref{torsione6},  the torsion of the canonical
connection $\nabla^{c}$  of the contact metric $(\kappa,\mu)$-space
$(M,\phi,\xi,\eta,g)$ satisfies $T^{c}(X,X')=2d\eta(X,X')\xi=0$.
Then, using \eqref{condizioniobata01}--\eqref{condizioniobata03}, we
have
\begin{align*}
{\phi'}_{1}[X,X']&={\phi'}_{1}\left(\nabla^{c}_{X}X'-\nabla^{c}_{X'}X\right)\\
&=\frac{1}{\sqrt{\alpha^2-\beta^2}}\left(\alpha\phi_{1}\nabla^{c}_{X}X'+\beta\phi_{3}\nabla^{c}_{X}X'-\alpha\phi_{1}\nabla^{c}_{X'}X-\beta\phi_{3}\nabla^{c}_{X'}X\right)\\
&=\frac{1}{\sqrt{\alpha^2-\beta^2}}\left(\alpha\nabla^{c}_{X}\phi_{1}X'+\beta\nabla^{c}_{X}\phi_{3}X'-\alpha\nabla^{c}_{X'}\phi_{1}X-\beta\nabla^{c}_{X'}\phi_{3}X\right)\\
&=\nabla^{c}_{X}{\phi'}_{1}X'-\nabla^{c}_{X'}{\phi'}_{1}X\\
&=[X,X'].
\end{align*}
In the same way one can prove that also ${{\mathcal D}'}_{1}^{-}$ is
involutive. Thus we conclude that the almost bi-paracontact
structure $({\phi'}_{1},{\phi'}_{2},{\phi'}_{3})$ is integrable. \
The case $|I_M|<1$ can be proved in  a similar way.
\end{proof}

\begin{remark}
By a straightforward computation one obtains
\begin{equation*}
h'_{1}=-{\sqrt{{I_{M}}^2-1}}h, \ \ \ h'_{2}=I_{M}\phi
h+\sqrt{1-\kappa}\phi, \ \ \ h'_{3}=0,
\end{equation*}
\begin{equation*}
h''_{1}=I_{M}\phi h+\sqrt{1-\kappa}\phi, \ \ \ h''_{2}=0, \ \ \
h''_{3}=\sqrt{1-{I_{M}}^2}h.
\end{equation*}
Moreover, the integrability of the almost bi-paracontact structure
yields, by Corollary \ref{proprieta7}, $N^{(1)}_{\phi'_{3}}=0$
 on $\mathcal D$. On the other hand, for any
$X\in\Gamma({\mathcal D})$,
$N^{(1)}_{\phi'_{3}}(X,\xi)=-[X,\xi]-\phi'_{3}[\phi'_{3}X,\xi]=2\phi'_{3}h'_{3}=0$.
Hence the almost contact structure $(\phi'_{3},\xi,\eta)$ is normal.
Nevertheless the almost bi-paracontact itself is not normal because
$h'_{1}$ and $h'_{2}$ do not vanish. Similar arguments hold for
$(\phi''_{1},\phi''_{2},\phi''_{3})$. Thus we have obtained a class
of examples of integrable, non-normal almost bi-paracontact
structures such that one structure is normal.
\end{remark}


\small

\begin{thebibliography}{99}
\bibitem{andrada} A.~Andrada, \textit{Complex product structures and affine
foliations}, Ann. Glob. Anal. Geom. \textbf{22} (2005), 377--405.

\bibitem{blair-1} D.~E.~Blair,  \textit{Two remarks on contact metric
structures}, T\^{o}hoku Math. J. \textbf{28} (1976), 373--379.

\bibitem{Blair-02} D.~E.~Blair, \textit{Riemannian geometry of
contact and symplectic manifolds}, Progress in Mathematics
\textbf{203}, Birkh\"{a}user, Boston, 2002.

\bibitem{BKP-95} D.~E.~Blair, T.~Koufogiorgos, B.~J.~Papantoniou,
\textit{Contact metric manifolds satisfyng a nullity condition},
Israel J. Math. \textbf{91} (1995), 189--214.

\bibitem{Boeckx-00} E.~Boeckx, \textit{A full classification of
contact metric $(\kappa,\mu )$-spaces}, Illinois J. Math.
\textbf{44} (2000), 212--219.

\bibitem{boeckx2}E. Boeckx, J. T. Cho,  \textit{$\eta$-parallel contact metric
spaces}, Different. Geom. Appl. \textbf{22} (2005), 275--285.

\bibitem{Mino-05} B.~Cappelletti Montano, \textit{Bi-Legendrian
connections}, Ann. Polon. Math. \textbf{86} (2005), 79--95.

\bibitem{Mino-07} B.~Cappelletti Montano, \textit{Some remarks on the
generalized Tanaka-Webster connection of a contact metric manifold},
Rocky Mountain J. Math., to appear.

\bibitem{Mino-Luigia-07} B.~Cappelletti Montano, L.~Di Terlizzi,
\textit{Contact metric $(\kappa,\mu)$-spaces as bi-Legendrian
manifolds}, Bull. Austral. Math. Soc. \textbf{77} (2008), 373--386.


\bibitem{Mino-08} B.~Cappelletti Montano, \textit{Bi-Legendrian structures and paracontact
geometry}, Int. J. Geom. Meth. Mod. Phys. \textbf{6} (2009),
487--504.

\bibitem{Mino-toap1} B.~Cappelletti Montano, \textit{The foliated structure of contact metric $(\kappa,\mu)$-spaces}, Illinois J. Math., to appear.

\bibitem{Mino-toap2} B.~Cappelletti Montano, L. Di Terlizzi, \textit{Geometric structures associated with a contact metric
$(\kappa,\mu)$-space}, submitted.

\bibitem{kaneyuki1} S.~Kaneyuki, F.~L.~Williams, \textit{Almost paracontact and parahodge structures on
manifolds}, Nagoya Math. J. \textbf{99} (1985), 173--187.

\bibitem{libermann} P.~Libermann, \textit{Legendre foliations on
contact manifolds}, Different. Geom. Appl. \textbf{1} (1991),
57--76.

\bibitem{morimoto1} Y.~Machida, T.~Morimoto, \textit{On decomposable Monge-Amp\`{e}re
equations}, Lobachevskii J. Math. \textbf{3} (1999), 185-196.

\bibitem{marchiafava} S.~Marchiafava, P.~T.~Nagy, \textit{(Anti-)hypercomplex structures and 3-webs on a
manifold}, Technical Report n. 38 (2003), University ``La Sapienza''
of Rome.

\bibitem{molino}P.~Molino, \textit{Riemannian foliations}, Progress
in Mathematics \textbf{73}, Birkh\"{a}user, Boston, 1988.

\bibitem{morimoto2} T.~Morimoto, \textit{Differential equations associated to a representation of a Lie algebra from the viewpoint of nilpotent analysis}, Developments of Cartan Geometry and Related Mathematical Problems, RIMS Kokyuroku  \textbf{1502} (2006), 238-250.

\bibitem{nagy} P.~T.~Nagy, \textit{Invariant tensorfields and the canonical connection of a
3-web}, Aequationes Math. \textbf{35} (1988), 31--44.


\bibitem{pang} M.~Y.~Pang, \textit{The structure of Legendre
foliations}, Trans. Amer. Math. Soc. \textbf{320} n. 2 (1990),
417--453.

\bibitem{santamaria}R.~Santamar\'{i}a~S\'{a}nchez, \textit{Examples of manifolds with three supplementary
distributions}, Atti Sem. Mat. Fis. Univ. Modena \textbf{157}
(1999), 419--428.

\bibitem{sasaki2}S.~Sasaki, Y.~Hatakeyama,  \textit{On differentiable manifolds with certain structures which are closely related to almost contact structure II}, T\^{o}hoku Math. J. \textbf{13} (1961), 281--294.

\bibitem{tanno} S.~Tanno, \textit{Variational problems on contact Riemannian
manifolds}, Trans. Amer. Math. Soc. \textbf{314} (1989), 349--379.

\bibitem{tondeur} Ph.~Tondeur, \textit{Geometry of foliations},
Monographs in Mathematics \textbf{90}, Birkh\"{a}user, Basel, 1997.


\bibitem{yano2} K.~Yano, M.~Ako, \textit{Almost quaternion structures of the second kind and almost tangent structures}, Kodai Math. Sem. Rep. \textbf{25} (1973), 63--94.

\bibitem{zamkovoy} S.~Zamkovoy, \textit{Canonical connections on paracontact
manifolds}, Ann. Glob. Anal. Geom. \textbf{36} (2009), 37--60.
\end{thebibliography}
\end{document}